\documentclass[12pt]{amsart}
\usepackage{fancybox}
\usepackage{shuffle}
\usepackage{amsmath}
\usepackage{amsfonts}
\usepackage{amssymb}
\usepackage[all]{xy}           %xypic macro for latex2.09
\usepackage{bm}
\usepackage{bbding}
\usepackage{txfonts}
\usepackage{amscd}

\usepackage{xspace}
\usepackage[shortlabels]{enumitem}
\usepackage{ifpdf}

\ifpdf
  \usepackage[colorlinks,final,backref=page,hyperindex]{hyperref}
\else
  \usepackage[colorlinks,final,backref=page,hyperindex,hypertex]{hyperref}
\fi
\usepackage{tikz}
\usepackage[active]{srcltx}

\makeatletter

\newtheorem{thm}{Theorem}[section]
\newtheorem{lem}[thm]{Lemma}
\newtheorem{cor}[thm]{Corollary}
\newtheorem{pro}[thm]{Proposition}
\newtheorem{ex}[thm]{Example}
\theoremstyle{definition}
\newtheorem{rmk}[thm]{Remark}
\newtheorem{defi}[thm]{Definition}

\newcommand{\nc}{\newcommand}
\newcommand{\delete}[1]{}

\nc{\mlabel}[1]{\label{#1}}  % Use this to suppress names
\nc{\mcite}[1]{\cite{#1}}  % Use this to suppress names
\nc{\mref}[1]{\ref{#1}}  % Use this to suppress names
\nc{\mbibitem}[1]{\bibitem{#1}} % Use this to show number
\delete{% Use the next two lines to show names
\nc{\mlabel}[1]{\label{#1}{\hfill \hspace{1cm}{\bf{{\ }\hfill(#1)}}}}
\nc{\mcite}[1]{\cite{#1}{{\em{{\ }(#1)}}}}  % Use this lines to show names
\nc{\mref}[1]{\ref{#1}{{\em{{\ }(#1)}}}}  % Use this lines to show names
\nc{\mbibitem}[1]{\bibitem[\em #1]{#1}} % Use this to show name
}

\newcommand {\emptycomment}[1]{}

\nc{\oprn}{\theta}
\nc{\Oprn}{\Theta}

\nc{\calo}{\mathcal{O}}
\nc{\oop}{$\mathcal{O}$-operator\xspace}
\nc{\oops}{$\mathcal{O}$-operators\xspace}
\nc{\mrho}{{\bm{\varrho}}}
\nc{\emk}{\mathbf{K}}
\nc{\invlim}{\displaystyle{\lim_{\longleftarrow}}\,}
\nc{\ot}{\otimes}

\newcommand{\be }{\begin{equation}}
\newcommand{\ee }{\end{equation}}

\newcommand{\g}{\mathfrak g}
\newcommand{\h}{\mathfrak h}
\newcommand{\m}{\mathfrak m}

\newcommand{\huaB}{\mathcal{B}}

\newcommand{\huaH}{\mathcal{H}}

\newcommand{\huaO}{{\mathcal{O}}}

\newcommand{\huaZ}{\mathcal{Z}}

\newcommand{\frkC}{\mathfrak C}

\newcommand{\frkH}{\mathfrak H}

\newcommand{\frkL}{\mathfrak L}

\newcommand{\frkR}{\mathfrak R}

\newcommand{\frkX}{\mathfrak X}

\newcommand{\Courant}[1]{\left\llbracket  #1\right\rrbracket }

\newcommand{\Id}{{\rm{Id}}}

\newcommand{\br}[1]{   [ \cdot,    \cdot  ]   }
\newcommand{\ltp}[1]{\Courant{\cdot,\cdot,\cdot}}

\newcommand{\Hom}{\mathrm{Hom}}

\newcommand{\Ob}{\mathsf{Ob}}

\newcommand{\gl}{\mathfrak {gl}}

\newcommand{\ad}{\mathrm{ad}}

\nc{\CV}{\mathbf{C}}

\newcommand{\LYA}{Lie-Yamaguti algebra}

\begin{document}

\title[Cohomology and deformations of crossed homomorphisms]{Cohomology and deformations of crossed homomorphisms between Lie-Yamaguti algebras}

\author{Jia Zhao}
\address{Jia Zhao, School of Sciences, Nantong University, Nantong, 226019, Jiangsu, China}
\email{zhaojia@ntu.edu.cn}

\author{Yu Qiao}
\address{Yu Qiao, School of Mathematics and Statistics, Shaanxi Normal University, Xi'an, 710119, Shaanxi, China}
\email{yqiao@snnu.edu.cn}

\author{Senrong Xu*}
\address{Senrong Xu (corresponding author), School of Mathematical Sciences, Jiangsu University, Zhenjiang, 212013, Jiangsu, China}
\email{senrongxu@ujs.edu.cn}

\date{\today}

%\thanks{{\em Mathematics Subject Classification} (2020): Primary 17B38; Secondary 17B60,17A99}

%\thanks{Qiao was partially supported by NSFC grant 11971282.}

\begin{abstract}
 In this paper, we introduce the notion of crossed homomorphisms between Lie-Yamaguti algebras and establish the cohomology theory of crossed homomorphisms via the Yamaguti cohomology. Consequently, we use this cohomology to characterize linear deformations of crossed homomorphisms between Lie-Yamaguti algebras. We show that if two linear or formal deformations of a crossed homomorphism are equivalent, then their infinitesimals are in the same cohomology class in the first cohomology group. Moreover, we show that an order $n$ deformation of a crossed homomorphism can be extended  to an order $n+1$ deformation if and only if the obstruction class in the second cohomology group is trivial.
\end{abstract}

%\subjclass[2020]{17B38,81R50,17B56,81R12,16T26,17A30,17B62}

%\keywords{Lie-Yamaguti algebra, crossed homomorphism, cohomology, deformation}

\maketitle

\keywords{{\em Keywords}: Lie-Yamaguti algebra, crossed homomorphism, cohomology, deformation}

\smallskip
%\vspace{-1.1cm}

\subjclass{\em Mathematics Subject Classification} (2020): 17A36, 17B55, 17B60

\tableofcontents

\allowdisplaybreaks

 \section{Introduction}
 Mathematical physics has many branches in mathematics and can be applied in Lie theory and representation theory \cite{IK,LiuPeiXia}. Deformation theory plays an important role in mathematics and mathematical physics. A deformation of a mathematical object, roughly speaking, means that it preserves its original structure after a parameter perturbation. In physics, deformation theory comes from quantizing classical mechanics, and this idea promotes some researches on quantum groups in mathematics \cite{CP,Hart}. Recently, deformation quantization has produced many elegant works in the context of mathematical physics. Based on work of complex analysis by Kodaira and Spencer \cite{Kodaira}, deformation theory was generalized in algebra \cite{Hart}. Deformation of algebra can be dated back to works on associative algebra by Gerstenhaber \cite{Gerstenhaber1,Gerstenhaber2,Gerstenhaber3,Gerstenhaber4,Gerstenhaber5}. Later, Nijenhuis and Richadson studied deformations on Lie algebra \cite{Nij1}. Balavoine generalized deformation theory to operads \cite{bala}.

 In the context of algebras, deformation has close connection with cohomology. For instance, a suitable cohomology can be used to characterize deformations. In particular, a linear deformation of Lie algebras is controlled by a second cohomolpogy group; an order $n$ deformation can be extended to an order $n+1$ deformation if and only if its obstruction class is trivial; a trivial deformation gives rise to a Nijenhuis operator \cite{Dor}, which is key to deformation theory and has applications in integrability of constructing biHamiltonian systems \cite{Dor}. Sheng and his collaborators have a series of woks on deformation theory on (3-)Lie algebras. For example, they studied deformations on 3-Lie algebras and even $n$-Lie algebras \cite{LSZB} and examined product and complex structures on $3$-Lie algebras using Nijenhuis operators \cite{T.S}. Moreover, they construct a controlling algebra that characterizes deformations of relative Rota-Baxter operators (also called $\huaO$-operators) on Lie algebras,  on $3$-Lie algebras, and on Leibniz algebras respectively \cite{TBGS,THS,T.S2}. Recently, Pei and his colleagues established crossed homomorphisms on Lie algebras via the same methods, and generalized constructions of many kinds of Lie algebras by using bifunctors \cite{PSTZ}. Besides, the first two authors investigated cohomology and linear deformations of $\mathsf{LieYRep}$ pairs and explored several properties of relative Rota-Baxter-Nijenhuis structures on $\mathsf{LieYRep}$ pairs in \cite{ZQ1}, and cohomology and deformations of relative Rota-Baxter operators on \LYA s in \cite{ZQ2}.

 \subsection{Lie-Yamaguti algebras}
 A Lie-Yamaguti algebra is a generalization of a Lie algebra and a Lie triple system which can be traced back to Nomizu's work on the affine invariant connections on homogeneous spaces in 1950's \cite{Nomizu}. Later in 1960's, Yamaguti introduced an algebraic structure and called it a general Lie triple system or a Lie triple algebra \cite{Yamaguti1,Yamaguti2,Yamaguti3}. Kinyon and Weinstein first called this object a \LYA~  when studying Courant algebroids in the earlier 21st century \cite{Weinstein}. Since then, this system was called a \LYA, which has attracted much attention and is widely investigated recently. For instance, Benito and his collaborators deeply explored irreducible Lie-Yamaguti algebras and their relations with orthogonal Lie algebras \cite{B.B.M,B.D.E,B.E.M1,B.E.M2}. Deformations and extensions of Lie-Yamaguti algebras were examined in \cite{L.CHEN,Ma Y,Zhang1,Zhang2}. Sheng, the first author, and Zhou analyzed product structures and complex structures on Lie-Yamaguti algebras by means of Nijenhuis operators in \cite{Sheng Zhao}. Takahashi studied modules over quandles using representations of Lie-Yamaguti algebras in \cite{Takahashi}.

 \subsection{Crossed homomorphisms}
 The notion of crossed homomorphisms on Lie algebras was introduced when nonabelian extension of Lie algebras was studied \cite{Lue}. An example of crossed homomorphisms is a differential operator of weight $1$, and a flat connection 1-form of a trivial principle bundle is also a crossed homomorphism. In \cite{PSTZ}, authors showed that the category of weak representations (resp. admissible representations) of Lie-Rinehart algebras (resp. Leibniz pairs) is a left module category over the monoidal category of representations of Lie algebras using crossed homomorphisms. Later, cohomology and deformations of crossed homomorphisms on $3$-Lie algebras were also studied in \cite{HHSZ}. Thus it is natural to consider cohomology and deformations of crossed homomorphisms between Lie-Yamaguti algebras.

 More precisely, for a crossed homomorphism $H:\g\longrightarrow\h$ from a \LYA ~$\g$ to another \LYA ~$\h$ with respect to an action $(\rho,\mu)$, the most important step is to establish the cohomology theory of $H$. Our strategy is as follows:  First we introduce linear maps $\rho_H:\g\longrightarrow\h$ and $\mu_H:\otimes^2\g\longrightarrow\h$ via $H$, and prove that $(\h;\rho_H,\mu_H)$ is a representation of $\g$ on the vector space $\h$. Consequently, we obtain a corresponding Yamaguti cohomology of \LYA ~$\g$ with coefficients in the representation $(\h;\rho_H,\mu_H)$. Note that Yamaguti cohomology stars from $1$-cochains. Thus the second step is to construct $0$-cochains and the corresponding coboundary maps, which is a difficulty to overcome. Once the cohomology theory is established,  we are able to explore the relationship between cohomology and deformations of crossed homomorphisms. For this purpose, we intend to investigate three kinds of deformations: linear, formal, and higher order deformations.

 Note that a \LYA ~can be reduced to a Lie triple system when the ternary bracket is trivial, thus the notion of crossed homomorphisms and the cohomology and deformation theory of those between Lie triple systems can be obtained directly from the present paper.

\subsection{Outline of the paper}
The paper is structured as follows. In Section 2, we recall some basic notions such as \LYA s, representations, and cohomology. In Section 3, we introduce the notion of crossed homomorphisms between \LYA s, and show that there is a one-to-one correspondence between crossed homomorphisms and \LYA ~homomorphisms. In Section 4, we establish the cohomology of crossed homomorphisms on \LYA s, and examine a functorial property of the cohomology theory. Finally in Section 5, we explore three kinds of deformations, and show that the infinitesimal of linear and formal deformations can be governed by cohomology and that the extension of a higher deformation is characterized by a special cohomology class. This is why we call this cohomology class the obstruction class.

In this paper, all vector spaces are assumed to be over a field $\mathbb{K}$ of characteristic $0$ and finite-dimensional.

\section{Preliminaries: Lie-Yamaguti algebras, representations and cohomology}
In this section, we recall some basic notions such as \LYA s, representations and their cohomology theories.
The notion of  Lie-Yamaguti algebras was introduced by  Yamaguti in \cite{Yamaguti1}.

\begin{defi}\cite{Weinstein}\label{LY}
A {\bf \LYA} is a vector space $\g$ equipped with a bilinear bracket $[\cdot,\cdot]:\wedge^2  \mathfrak{g} \to \mathfrak{g} $ and a trilinear bracket $\Courant{\cdot,\cdot,\cdot}:\wedge^2\g \otimes  \mathfrak{g} \to \mathfrak{g} $, which meet the following conditions: for all $x,y,z,w,t \in \g$,
\begin{eqnarray}
~ &&\label{LY1}[[x,y],z]+[[y,z],x]+[[z,x],y]+\Courant{x,y,z}+\Courant{y,z,x}+\Courant{z,x,y}=0,\\
~ &&\label{LY2}\Courant{[x,y],z,w}+\Courant{[y,z],x,w}+\Courant{[z,x],y,w}=0,\\
~ &&\label{LY3}\Courant{x,y,[z,w]}=[\Courant{x,y,z},w]+[z,\Courant{x,y,w}],\\
~ &&\Courant{x,y,\Courant{z,w,t}}=\Courant{\Courant{x,y,z},w,t}+\Courant{z,\Courant{x,y,w},t}+\Courant{z,w,\Courant{x,y,t}}.\label{fundamental}
\end{eqnarray}
In the sequel, we denote a \LYA ~by $(\g,\br,\ltp))$.
\end{defi}

\begin{ex}
Let $(\g,\br))$ be a Lie algebra. Define a trilinear bracket
$$\ltp::\wedge^2\g\ot \g\to \g$$
by
$$\Courant{x,y,z}:=[[x,y],z],\quad \forall x,y,z \in \g.$$
Then by a direct computation, we know that $(\g,\br,\ltp))$ forms a \LYA.
\end{ex}

The following example is even more interesting.
\begin{ex}
Let $M$ be a closed manifold with an affine connection, and denote by $\frkX(M)$ the set of vector fields on $M$. For all $x, y, z \in \frkX(M) $, set
\begin{eqnarray*}
[x,y]&:=&-T(x,y),\\
\Courant{x,y,z}&:=&-R(x,y)z,
\end{eqnarray*}
where $T$ and $R$ are torsion tensor and curvature tensor respectively. It turns out that the triple
$ (\frkX(M),[\cdot,\cdot],\Courant{\cdot,\cdot,\cdot})$ forms a \LYA. See \cite{Nomizu} for more details.
\end{ex}
\emptycomment{
\begin{rmk}
Given a Lie-Yamaguti algebra $(\m,[\cdot,\cdot]_\m,\Courant{\cdot,\cdot,\cdot}_\m)$ and any two elements $x,y \in \m$, the linear map $D(x,y):\m \to \m,~z\mapsto D(x,y)z=\Courant{x,y,z}_\m$ is an (inner) derivation. Moreover, let $D(\m,\m)$ be the linear span of the inner derivations. Consider the vector space $\g(\m)=D(\m,\m)\oplus \m$, and endow it with a Lie bracket as follows: for all $x,y,z,t \in \m$
\begin{eqnarray*}
[D(x,y),D(z,t)]_{\g(\m)}&=&D(\Courant{x,y,z}_\m,t)+D(z,\Courant{x,y,t}_\m),\\
~[D(x,y),z]_{\g(\m)}&=&D(x,y)z=\Courant{x,y,z}_\m,\\
~[z,t]_{\g(\m)}&=&D(z,t)+[z,t]_\m.
\end{eqnarray*}
Then $(\g(\m),[\cdot,\cdot]_{\g(\m)})$ becomes a Lie algebra.
\end{rmk}}

Next, we recall the notion of representations of \LYA s.

\begin{defi}\cite{Yamaguti2}\label{defi:representation}
Let $(\g,[\cdot,\cdot],\Courant{\cdot,\cdot,\cdot})$ be a Lie-Yamaguti algebra. A {\bf representation} of $\g$ is a vector space $V$ equipped with a linear map $\rho:\g \to \gl(V)$ and a bilinear map $\mu:\otimes^2 \g \to \gl(V)$, which meet the following conditions: for all $x,y,z,w \in \g$,
\begin{eqnarray}
~&&\label{RLYb}\mu([x,y],z)-\mu(x,z)\rho(y)+\mu(y,z)\rho(x)=0,\\
~&&\label{RLYd}\mu(x,[y,z])-\rho(y)\mu(x,z)+\rho(z)\mu(x,y)=0,\\
~&&\label{RLYe}\rho(\Courant{x,y,z})=[D_{\rho,\mu}(x,y),\rho(z)],\\
~&&\label{RYT4}\mu(z,w)\mu(x,y)-\mu(y,w)\mu(x,z)-\mu(x,\Courant{y,z,w})+D_{\rho,\mu}(y,z)\mu(x,w)=0,\\
~&&\label{RLY5}\mu(\Courant{x,y,z},w)+\mu(z,\Courant{x,y,w})=[D_{\rho,\mu}(x,y),\mu(z,w)],
\end{eqnarray}
where the bilinear map $D_{\rho,\mu}:\otimes^2\g \to \gl(V)$ is given by
\begin{eqnarray}
 D_{\rho,\mu}(x,y):=\mu(y,x)-\mu(x,y)+[\rho(x),\rho(y)]-\rho([x,y]), \quad \forall x,y \in \g.\label{rep}
 \end{eqnarray}
It is obvious that $D_{\rho,\mu}$ is skew-symmetric, and we write $D$ in the sequel without ambiguities.  We denote a representation of $\g$ by $(V;\rho,\mu)$.
\end{defi}

\begin{rmk}\label{rmk:rep}
Let $(\g,[\cdot,\cdot],\Courant{\cdot,\cdot,\cdot})$ be a Lie-Yamaguti algebra and $(V;\rho,\mu)$ a representation of $\g$. If $\rho=0$ and the Lie-Yamaguti algebra $\g$ reduces to a Lie tripe system $(\g,\Courant{\cdot,\cdot,\cdot})$,  then the representation reduces to that of the Lie triple system $(\g,\Courant{\cdot,\cdot,\cdot})$: $(V;\mu)$. If $\mu=0$, $D=0$ and the Lie-Yamaguti algebra $\g$ reduces to a Lie algebra $(\g,[\cdot,\cdot])$, then the representation reduces to that of the Lie algebra $(\g,[\cdot,\cdot])$: $(V;\rho)$. Hence a representation of a Lie-Yamaguti algebra is a natural generalization of that of a Lie algebra or of a Lie triple system.
\end{rmk}

By a direct computation, we have the following lemma.
\begin{lem}
Suppose that $(V;\rho,\mu)$ is a representation of a Lie-Yamaguti algebra $(\g,[\cdot,\cdot],\Courant{\cdot,\cdot,\cdot})$. Then the following equalities are satisfied:
\begin{eqnarray*}
\label{RLYc}&&D([x,y],z)+D([y,z],x)+D([z,x],y)=0;\\
\label{RLY5a}&&D(\Courant{x,y,z},w)+D(z,\Courant{x,y,w})=[D(x,y),D_{\rho,\mu}(z,w)];\\
~ &&\mu(\Courant{x,y,z},w)=\mu(x,w)\mu(z,y)-\mu(y,w)\mu(z,x)-\mu(z,w)D(x,y).\label{RLY6}
\end{eqnarray*}
\end{lem}

\begin{ex}\label{ad}
Let $(\g,[\cdot,\cdot],\Courant{\cdot,\cdot,\cdot})$ be a Lie-Yamaguti algebra. We define linear maps $\ad:\g \to \gl(\g)$ and $\frkR :\otimes^2\g \to \gl(\g)$ by $x \mapsto \ad_x$ and $(x,y) \mapsto \mathfrak{R}_{x,y}$ respectively, where $\ad_xz=[x,z]$ and $\mathfrak{R}_{x,y}z=\Courant{z,x,y}$ for all $z \in \g$. Then $(\ad,\mathfrak{R})$ forms a representation of $\g$ on itself, where $\frkL:= D_{\ad,\frkR}$ is given by
\begin{eqnarray*}
\frkL_{x,y}=\mathfrak{R}_{y,x}-\mathfrak{R}_{x,y}+[\ad_x,\ad_y]-\ad_{[x,y]}, \quad \forall x,y \in \g.
\end{eqnarray*}
By \eqref{LY1}, we have
\begin{eqnarray*}
\frkL_{x,y}z=\Courant{x,y,z}, \quad \forall z \in \g.\label{lef}
\end{eqnarray*}
In this case, $(\g;\ad,\frkR)$ is called the {\bf adjoint representation} of $\g$.
\end{ex}
\emptycomment{
The representations of Lie-Yamaguti algebras can be characterized by the semidirect Lie-Yamaguti algebras. This fact is revealed via the following proposition.

\begin{pro}\cite{Zhang1}
Let $(\g,[\cdot,\cdot],\Courant{\cdot,\cdot,\cdot})$ be a Lie-Yamaguti algebra and $V$ a vector space. Let $\rho:\g \to \gl(V)$ and $\mu:\otimes^2 \g \to \gl(V)$ be linear maps. Then $(V;\rho,\mu)$ is a representation of $(\g,[\cdot,\cdot],\Courant{\cdot,\cdot,\cdot})$ if and only if there is a Lie-Yamaguti algebra structure $([\cdot,\cdot]_{\rho,\mu},\Courant{\cdot,\cdot,\cdot}_{\rho,\mu})$ on the direct sum $\g \oplus V$ which is defined by for all $x,y,z \in \g, ~u,v,w \in V$,
\begin{eqnarray}
\label{semi1}[x+u,y+v]_{\rho,\mu}&=&[x,y]+\rho(x)v-\rho(y)u,\\
\label{semi2}~\Courant{x+u,y+v,z+w}_{\rho,\mu}&=&\Courant{x,y,z}+D_{\rho,\mu}(x,y)w+\mu(y,z)u-\mu(x,z)v,
\end{eqnarray}
where $D_{\rho,\mu}$ is given by \eqref{rep}.
This Lie-Yamaguti algebra $(\g \oplus V,[\cdot,\cdot]_{\rho,\mu},\Courant{\cdot,\cdot,\cdot}_{\rho,\mu})$ is called the {\bf semidirect product Lie-Yamaguti algebra}, and is denoted by $\g \ltimes_{\rho,\mu} V$.
\end{pro}
\begin{proof}
The proof is a direct computation, so we omit the details. Or one can see \cite{Zhang1} for more details.
\end{proof}}

Let us recall the cohomology theory on Lie-Yamaguti algebras given in \cite{Yamaguti2}. Let $(\g,[\cdot,\cdot],\Courant{\cdot,\cdot,\cdot})$ be a  Lie-Yamaguti algebra and $(V;\rho,\mu)$ a representation of $\g$. We denote the set of $p$-cochains by $C^p_{\rm LieY}(\g,V)~(p \geqslant 1)$, where
\begin{eqnarray*}
C^{n+1}_{\rm LieY}(\g,V)\triangleq
\begin{cases}
\Hom(\underbrace{\wedge^2\g\otimes \cdots \otimes \wedge^2\g}_n,V)\times \Hom(\underbrace{\wedge^2\g\otimes\cdots\otimes\wedge^2\g}_{n}\otimes\g,V), & \forall n\geqslant 1,\\
\Hom(\g,V), &n=0.
\end{cases}
\end{eqnarray*}

In the sequel, we recall the coboundary map of $p$-cochains:
\begin{itemize}
\item If $n\geqslant 1$, for any $(f,g)\in C^{n+1}_{\rm LieY}(\g,V)$, the coboundary map
$$\delta=(\delta_{\rm I},\delta_{\rm II}):C^{n+1}_{\rm LieY}(\g,V)\to C^{n+2}_{\rm LieY}(\g,V),$$
$$\qquad \qquad\qquad \qquad\qquad \quad (f,g)\mapsto(\delta_{\rm I}(f,g),\delta_{\rm II}(f,g))$$
 is given as follows:
\begin{eqnarray}
~\nonumber &&\Big(\delta_{\rm I}(f,g)\Big)(\frkX_1,\cdots,\frkX_{n+1})\\
~ &=&(-1)^n\Big(\rho(x_{n+1})g(\frkX_1,\cdots,\frkX_n,y_{n+1})-\rho(y_{n+1})g(\frkX_1,\cdots,\frkX_n,x_{n+1})\label{cohomology1}\\
~\nonumber &&\quad\quad-g(\frkX_1,\cdots,\frkX_n,[x_{n+1},y_{n+1}])\Big)\\
~\nonumber &&+\sum_{k=1}^{n}(-1)^{k+1}D(\frkX_k)f(\frkX_1,\cdots,\widehat{\frkX_k},\cdots,\frkX_{n+1})\\
~\nonumber &&+\sum_{1\leqslant k<l\leqslant n+1}(-1)^{k}f(\frkX_1,\cdots,\widehat{\frkX_k},\cdots,\frkX_k\circ\frkX_l,\cdots,\frkX_{n+1}),
\end{eqnarray}
\begin{eqnarray}
~\nonumber &&\Big(\delta_{\rm II}(f,g)\Big)(\frkX_1,\cdots,\frkX_{n+1},z)\\
~ &=&(-1)^n\Big(\mu(y_{n+1},z)g(\frkX_1,\cdots,\frkX_n,x_{n+1})-\mu(x_{n+1},z)g(\frkX_1,\cdots,\frkX_n,y_{n+1})\Big)\label{cohomology2}\\
~\nonumber &&+\sum_{k=1}^{n+1}(-1)^{k+1}D(\frkX_k)g(\frkX_1,\cdots,\widehat{\frkX_k},\cdots,\frkX_{n+1},z)\\
~\nonumber &&+\sum_{1\leqslant k<l\leqslant n+1}(-1)^kg(\frkX_1,\cdots,\widehat{\frkX_k},\cdots,\frkX_k\circ\frkX_l,\cdots,\frkX_{n+1},z)\\
~\nonumber &&+\sum_{k=1}^{n+1}(-1)^kg(\frkX_1,\cdots,\widehat{\frkX_k},\cdots,\frkX_{n+1},\Courant{x_k,y_k,z}),
\end{eqnarray}
where $\frkX_i=x_i\wedge y_i\in\wedge^2\g~(i=1,\cdots,n+1),~z\in \g$ and $\frkX_k\circ\frkX_l:=\Courant{x_k,y_k,x_l}\wedge y_l+x_l\wedge\Courant{x_k,y_k,y_l}$.

\item If $n=0$, for any $f \in C^1_{\rm LieY}(\g,V)$, the coboundary map
$$\delta:C^1_{\rm LieY}(\g,V)\to C^2_{\rm LieY}(\g,V),$$
$$\qquad \qquad \qquad f\mapsto (\delta_{\rm I}(f),\delta_{\rm II}(f))$$
is defined to be
\begin{eqnarray}
\label{cohomology3}\Big(\delta_{\rm I}(f)\Big)(x,y)&=&\rho(x)f(y)-\rho(y)f(x)-f([x,y]),\\
~ \label{cohomology4}\Big(\delta_{\rm II}(f)\Big)(x,y,z)&=&D(x,y)f(z)+\mu(y,z)f(x)-\mu(x,z)f(y)-f(\Courant{x,y,z}),\quad \forall x,y, z\in \g.
\end{eqnarray}
\end{itemize}

Yamaguti showed the following fact.

\begin{pro}{\rm \cite{Yamaguti2}}
 With the notations above, for any $f\in C^1_{\rm LieY}(\g,V)$, we have
 \begin{eqnarray*}
 \delta_{\rm I}\Big(\delta_{\rm I}(f)),\delta_{\rm II}(f)\Big)=0\quad {\rm and} \quad\delta_{\rm II}\Big(\delta_{\rm I}(f)),\delta_{\rm II}(f)\Big)=0.
 \end{eqnarray*}
 Moreover, for all $(f,g)\in C^p_{\rm LieY}(\g,V),~(p\geqslant 2)$, we have
  \begin{eqnarray*}
  \delta_{\rm I}\Big(\delta_{\rm I}(f,g)),\delta_{\rm II}(f,g)\Big)=0\quad{\rm and} \quad \delta_{\rm II}\Big(\delta_{\rm I}(f,g)),\delta_{\rm II}(f,g)\Big) =0.
  \end{eqnarray*}
  Thus the cochain complex $(C^\bullet_{\rm LieY}(\g,V)=\bigoplus\limits_{p=1}^\infty C^p_{\rm LieY}(\g,V),\delta)$ is well defined. For convenience, we call this cohomology the {\bf Yamaguti cohomology} in this paper.
  \end{pro}

\begin{defi}
With the above notations, let $(f,g)$ in $C^p_{\rm LieY}(\g,V))$ (resp. $f\in C^1_{\rm LieY}(\g,V)$ for $p=1$) be a $p$-cochain. If it satisfies $\delta(f,g)=0$ (resp. $\delta(f)=0$), then it is called a $p$-cocycle. If there exists $(h,s)\in C^{p-1}_{\rm LieY}(\g,V)$,~(resp. $t\in C^1(\g,V)$, if $p=2$) such that $(f,g)=\delta(h,s)$~(resp. $(f,g)=\delta(t)$), then it is called a $p$-coboundary ($p\geqslant 2$). The set of $p$-cocycles and that of $p$-coboundaries are denoted by $Z^p_{\rm LieY}(\g,V)$ and $B^p_{\rm LieY}(\g,V)$ respectively. The resulting $p$-cohomology group is defined to be the factor space
$$H^p_{\rm LieY}(\g,V)=Z^p_{\rm LieY}(\g,V)/B^p_{\rm LieY}(\g,V).$$ In particular, we have
$$H^1_{\rm LieY}(\g,V)=\{f\in C^1_{\rm LieY}(\g,V):\delta (f)=0\}.$$
\end{defi}

\section{Crossed homomorphisms between Lie-Yamaguti algebras}
In this section, we introduce the notion of crossed homomorphisms between \LYA s, and show that a crossed homomorphism can be seen as a homomorphism of \LYA s. Moreover, a crossed homomorphism corresponds to a relative Rota-Baxter operator of weight 1.
Before this, we introduce the notion of center of \LYA s.

Let $(\g,[\cdot,\cdot]_\g,\Courant{\cdot,\cdot}_\g)$ be a Lie-Yamaguti algebra. Denote the {\bf center} of $\g$ by
$$C(\g):=\{x\in \g|[x,y]=0,\forall y\in \g\}\cap\Big(\{x\in \g|\Courant{x,y,z}=0,\forall y,z \in \g\}\cup\{x\in \g|\Courant{y,x,z}=0,\forall y,z \in \g\}\Big).$$

\begin{defi}
Let $(\g,[\cdot,\cdot]_\g,\Courant{\cdot,\cdot}_\g)$ and $(\h,[\cdot,\cdot]_\h,\Courant{\cdot,\cdot}_\h)$ be two Lie-Yamaguti algebras. Let $(\h;\rho,\mu)$ be a representation of $\g$ on the vector space $\h$, i.e., linear maps $\rho:\g\to\gl(\h)$, $\mu:\otimes^2\g\to\gl(\h)$, and $D:\wedge^2\g\to\gl(\g)$ are given by Eqs. \eqref{RLYb}-\eqref{rep}. If for all $x,y\in \g,~u,v,w\in \h$, the following conditions are satisfied
\begin{eqnarray*}
\rho(x)u,\mu(x,y)u\in C(\h),\\
\rho(x)[u,v]_\h=\mu(x,y)[u,v]_\h=0,\\
\rho(x)\Courant{u,v,w}_\h=\mu(x,y)\Courant{u,v,w}_\h=0.
\end{eqnarray*}
then we say that $(\rho,\mu)$ is an {\bf action} of $\g$ on $\h$.
\end{defi}

Let $(\rho,\mu)$ be an action of $\g$ on $\h$. By \eqref{rep}, we deduce that
$$D(x,y)u\in C(\g),~D(x,y)[u,v]_\h=D(x,y)\Courant{u,v,w}_\h=0,\quad\forall x,y\in \g,u,v,w\in \h.$$

The following proposition shows that an action of \LYA s can be used to characterize semidirect product \LYA s.
\begin{pro}
Let $(\g,[\cdot,\cdot]_\g,\Courant{\cdot,\cdot}_\g)$ and $(\h,[\cdot,\cdot]_\h,\Courant{\cdot,\cdot}_\h)$ be two Lie-Yamaguti algebras. Let $(\rho,\mu)$ be an action of $\g$ on $\h$, then there is a Lie-Yamaguti algebra structure on the direct sum $\g\oplus\h$ defined by
\begin{eqnarray*}
[x+u,y+v]_{\rho,\mu}&=&[x,y]_\g+\rho(x)v-\rho(y)u+[u,v]_\h,\\
\Courant{x+u,y+v,z+w}_{\rho,\mu}&=&\Courant{x,y,z}_\g+D(x,y)w+\mu(y,z)u-\mu(x,z)v+\Courant{u,v,w}_\h,
\end{eqnarray*}
for all $x,y,z\in \g$ and $u,v,w\in \h$. This Lie-Yamaguti algebra is called the {\bf semidirect product Lie-Yamaguti algebra} with respect to the action $(\rho,\mu)$, and is denoted by $\g\ltimes_{\rho,\mu}\h$.
\end{pro}
\begin{proof}
It is a direct computation, and we omit the details.
\end{proof}

The following definition is standard.

\begin{defi}\cite{Sheng Zhao,Takahashi}\label{homomorphism}
Suppose that $(\g,[\cdot,\cdot]_{\g},\Courant{\cdot,\cdot,\cdot}_{\g})$ and $(\h,[\cdot,\cdot]_{\h},\Courant{\cdot,\cdot,\cdot}_{\h})$ are two Lie-Yamaguti algebras. A {\bf homomorphism} from $(\g,[\cdot,\cdot]_{\g},\Courant{\cdot,\cdot,\cdot}_{\g})$ to $(\h,[\cdot,\cdot]_{\h},\Courant{\cdot,\cdot,\cdot}_{\h})$ is a linear map $\phi:\g \to \h$ that preserves the \LYA ~structures, that is, for all $x,y,z \in \g$,
\begin{eqnarray*}
\phi([x,y]_{\g})&=&[\phi(x),\phi(y)]_{\h},\\
~ \phi(\Courant{x,y,z}_{\g})&=&\Courant{\phi(x),\phi(y),\phi(z)}_{\h}.
\end{eqnarray*}
If, moreover, $\phi$ is a bijection, it is then called an {\bf isomorphism}.
\end{defi}

Now we are ready to introduce the notion of crossed homomorphisms between \LYA s.

\begin{defi}
Let $(\g,[\cdot,\cdot]_\g,\Courant{\cdot,\cdot}_\g)$ and $(\h,[\cdot,\cdot]_\h,\Courant{\cdot,\cdot}_\h)$ be two Lie-Yamaguti algebras. Let $(\rho,\mu)$ be an action of $\g$ on $\h$. A linear map $H:\g\longrightarrow\h$ is called a {\bf crossed homomorphism} from $\g$ to $\h$ with respect to $(\rho,\mu)$, if
\begin{eqnarray}
H[x,y]_\g&=&\rho(x)H(y)-\rho(y)H(x)+[Hx,Hy]_\h,\label{chomo1}\\
\ \ H\Courant{x,y,z}_\g&=&D(x,y)H(z)+\mu(y,z)H(x)-\mu(x,z)H(y)+\Courant{Hx,Hy,Hz}_\h,\quad\forall x,y,z\in \g.\label{chomo2}
\end{eqnarray}
\end{defi}

\begin{rmk}
If the action of $\g$ on $\h$ is trivial, then any crossed homomorphism from $\g$ to $\h$ is a Lie-Yamaguti algebra homomorphism as in Definition \ref{homomorphism}; if $\h$ is commutative, then any crossed homomorphism is a derivation from $\g$ to $\h$ with respect to the representation $(\h;\rho,\mu)$.
\end{rmk}

\begin{ex}
Let $(\g,[\cdot,\cdot],\Courant{\cdot,\cdot,\cdot})$ be a $4$-dimensional \LYA, and $\{e_1,e_2,e_3,e_4\}$ a basis. The nonzero brackets are given by
$$[e_1,e_2]=2e_4,\quad\quad \Courant{e_1,e_2,e_1}=e_4.$$
It is obvious that the center of $\g$ is spanned by $\{e_3,e_4\}$, and that the adjoint representation $(\g;\ad,\frkR)$ is an action of $\g$ on itself. Then
\begin{eqnarray*}
H=
\begin{pmatrix}
O & A\\
B & C
\end{pmatrix}
\end{eqnarray*}
is a crossed homomorphism from $\g$ to $\g$, where $A$ and $C$ have the form
\begin{eqnarray*}
\begin{pmatrix}
\lambda_1 & 0\\
\lambda_2 & 0
\end{pmatrix}
.
\end{eqnarray*}
\end{ex}

The following theorem shows that a crossed homomorphism can be seen as a homomorphism between \LYA s.

\begin{thm}\label{thm1}
Let $(\rho,\mu)$ be an action of a Lie-Yamaguti algebra $(\g,[\cdot,\cdot]_\g,\Courant{\cdot,\cdot,\cdot}_\g)$ on another Lie-Yamaguti algebra $(\h,[\cdot,\cdot]_\h,\Courant{\cdot,\cdot,\cdot}_\h)$. Then a linear map $H:\g\longrightarrow\h$ is a crossed homomorphism from $\g$ to $\h$ if and only if the linear map $\phi_H:\g\longrightarrow\g\ltimes_{\rho,\mu}\h$ is a Lie-Yamaguti algebra homomorphism, where
$$\phi_H(x):=(x,Hx),\quad\forall x\in \g.$$
\end{thm}
\begin{proof}
For all $x,y,z\in \g$, we have
\begin{eqnarray*}
\phi_H\Big([x,y]_\g\Big)&=&\Big([x,y]_\g,H[x,y]_\g\Big),\\
~[\phi_H(x),\phi_H(y)]_{\rho,\mu}&=&[(x,Hx),(y,Hy)]_{\rho,\mu}=\Big([x,y]_\g,\rho(x)Hy-\rho(y)Hx+[Hx,Hy]_\h\Big).
\end{eqnarray*}

Similarly, we also have
\begin{eqnarray*}
\phi_H\Big(\Courant{x,y,z}]_\g\Big)&=&\Big(\Courant{x,y,z}_\g,H\Courant{x,y,z}_\g\Big),\\
~\Courant{\phi_H(x),\phi_H(y),\phi_H(z)}_{\rho,\mu}&=&\Big(\Courant{x,y,z}_\g,D(x,y)H(z)+\mu(y,z)H(x)-\mu(x,z)H(y)+\Courant{Hx,Hy,Hz}_\h\Big).
\end{eqnarray*}
Thus, we have that the linear map $\phi_H:\g\longrightarrow\g\ltimes_{\rho,\mu}\h$ is a Lie-Yamaguti algebra homomorphism if and only of the following two equalities hold:
\begin{eqnarray*}
H[x,y]_\g&=&\rho(x)Hy-\rho(y)Hx+[Hx,Hy]_\h,\label{cross1}\\
\ \ \ ~H\Courant{x,y,z}_\g&=&D(x,y)H(z)+\mu(y,z)H(x)-\mu(x,z)H(y)+\Courant{Hx,Hy,Hz}_\h,\quad\forall x,y,z\in \g,\label{cross2}
\end{eqnarray*}
which implies that the linear map $H:\g\longrightarrow\h$ is a crossed homomorphism from $\g$ to $\h$. This completes the proof.
\end{proof}

\begin{rmk}
In fact, a crossed homomorphism corresponds to a split nonabelian extension of Lie-Yamaguti algebras. More precisely, consider the following nonabelian extension of Lie-Yamaguti algebras:
\[
\xymatrix{
0 \ar[r] & \h \ar[r] & \g\oplus\h \ar[r] & \g \ar[r] & 0.}
\]
A section $s:\g\longrightarrow\g\oplus\h$ is given by $s(x)=(x,Hx), ~x\in \g$. Theorem \ref{thm1} says that $s$ is a Lie-Yamaguti algebra homomorphism if and only if $H$ is a crossed homomorphism. Such an extension is called a split nonabelian extension. See \cite{Zhang1} for more details about extension of \LYA s.
\end{rmk}

Then we introduce the notion of homomorphisms of crossed homomorphisms.

\begin{defi}
Let $H$ and $H'$ be two crossed homomorphisms from a Lie-Yamaguti algebra $(\g,[\cdot,\cdot]_\g,\Courant{\cdot,\cdot,\cdot}_\g)$ to another Lie-Yamaguti algebra $(\h,[\cdot,\cdot]_\h,\Courant{\cdot,\cdot,\cdot}_\h)$ with respect to an action $(\rho,\mu)$. A {\bf homomorphism} from $H'$ to $H$ is a pair $(\psi_\g,\psi_\h)$, where $\psi_\g:\g\longrightarrow\g$ and $\psi_\h:\h\longrightarrow\h$ are two Lie-Yamaguti algebra homomorphisms such that
\begin{eqnarray}
\psi_\h\circ H'&=&H\circ \psi_\g,\label{crho1}\\
\psi_\h\Big(\rho(x)u\Big)&=&\rho\Big(\psi_\g(x)\Big)\psi_\h(u),\label{crosshomo1}\\
\psi_\h\Big(\mu(x,y)u\Big)&=&\mu\Big(\psi_\g(x),\psi_\g(y)\Big)\psi_\h(u),\quad\forall x,y\in \g, u\in \h.\label{crosshomo2}
\end{eqnarray}
In particular, if both $\psi_\g$ and $\psi_\h$ are invertible, then $(\psi_\g,\psi_h)$ is called an {\bf isomorphism} from $H'$ to $H$.
\end{defi}

By Eqs. \eqref{crosshomo1} and \eqref{crosshomo2}, and a direct computation, we have the following proposition.

\begin{pro}
Let $H$ and $H'$ be two crossed homomorphisms from a Lie-Yamaguti algebra $(\g,[\cdot,\cdot]_\g,\Courant{\cdot,\cdot,\cdot}_\g)$ to another Lie-Yamaguti algebra $(\h,[\cdot,\cdot]_\h,\Courant{\cdot,\cdot,\cdot}_\h)$ with respect to an action $(\rho,\mu)$. Suppose that $(\psi_\g,\psi_\h)$ is a homomorphism from $H'$ to $H$, then we have
\begin{eqnarray}
\psi_\h\Big(D(x,y)u\Big)&=&D\Big(\psi_\g(x),\psi_\g(y)\Big)\psi_\h(u),\quad\forall x,y\in \g, u\in \h.\label{crosshomo3}
\end{eqnarray}
\end{pro}

At the end of this section, we reveal the relationship between crossed homomorphisms between \LYA s and relative Rota-Baxter operators of weight 1 on \LYA s. We give the notion of relative Rota-Baxter operators of weight $\lambda$ on \LYA s first.

\begin{defi}
Let $(\g,[\cdot,\cdot],\Courant{\cdot,\cdot,\cdot})$ be a \LYA ~and $(V;\rho,\mu)$ its representation. A linear map $T:V\longrightarrow\g$ is called {\bf a relative Rota-Baxter operator of weight $\lambda$} if the following equalities hold
\begin{eqnarray*}
[Tu,Tv]&=&T\Big(\rho(Tu)v-\rho(Tv)u+\lambda[u,v]\Big),\\
\Courant{Tu,Tv,Tw}&=&T\Big(D(Tu,Tv)w+\mu(Tv,Tw)u-\mu(Tu,Tw)v+\lambda\Courant{u,v,w}\Big),\quad\forall u,v,w\in V.
\end{eqnarray*}
\end{defi}

Relative Rota-Baxter operators of nonzero weight on Lie algebras stem from the classical Yang-Baxter equation and have many applications on mathematical physics. Here, we introduce the notion of relative Rota-Baxter  operators of nonzero weight on \LYA s and explore its relation with crossed homomorphisms.

\begin{pro}
Let $(\rho,\mu)$ be an action of a Lie-Yamaguti algebra $(\g,[\cdot,\cdot]_\g,\Courant{\cdot,\cdot,\cdot}_\g)$ on another Lie-Yamaguti algebra $(\h,[\cdot,\cdot]_\h,\Courant{\cdot,\cdot,\cdot}_\h)$. An invertible linear map $H:\g\longrightarrow\h$ is a crossed homomorphism from $\g$ to $\h$ with respect to $(\rho,\mu)$ if and only if $H^{-1}:\h\longrightarrow\g$ is a relative Rota-Baxter operator of weight $1$ on $\g$ with respect to the representation $(\h;\rho,\mu)$.
\end{pro}
\begin{proof}
Assume that the invertible linear map $H:\g\longrightarrow\h$ is a crossed homomorphism, then for all $u,v,w\in \h$, by \eqref{chomo1}, we have
\begin{eqnarray*}
~ &&[H^{-1}(u),H^{-1}(v)]_\g=H^{-1}\Big(H[H^{-1}(u),H^{-1}(v)]_\g\Big)\\
~ &=&H^{-1}\Big(\rho\big(H^{-1}(u)\big)v-\rho\big(H^{-1}(v)\big)w+[u,v]_\h\Big).
\end{eqnarray*}
Similarly, by \eqref{chomo2}, we have
\begin{eqnarray*}
~ &&\Courant{H^{-1}(u),H^{-1}(v),H^{-1}(w)}_\g\\
~ &=&H^{-1}\Big(H\Courant{H^{-1}(u),H^{-1}(v),H^{-1}(w)}_\g\Big)\\
~ &=&H^{-1}\Big(D\big(H^{-1}(u),H^{-1}(v)\big)w+\mu\big(H^{-1}(v),H^{-1}(w)\big)u-\mu\big(H^{-1}(u),H^{-1}(w)\big)v+\Courant{u,v,w}_\h\Big).
\end{eqnarray*}
Thus $H^{-1}$ is a relative Rota-Baxter operator of weight $1$.

Conversely, let $H^{-1}$ is a relative Rota-Baxter operator of weight $1$. For all $x,y,z\in \g$, there exist $u,v,w\in \h$, such that $x=H^{-1}(u),y=H^{-1}(v)$, and $z=H^{-1}(w)$. Then we have
\begin{eqnarray*}
~ &&H[x,y]_\g=H[H^{-1}(u),H^{-1}(v)]_\g\\
~ &=&H\circ H^{-1}\Big(\rho\big(H^{-1}(u)\big)v-\rho\big(H^{-1}(v)\big)u+[u,v]_\h\Big)\\
~ &=&\rho(x)H(y)-\rho(y)H(x)+[Hx,Hy]_\h,
\end{eqnarray*}
and
\begin{eqnarray*}
~ &&H\Courant{x,y,z}_\g=H[H^{-1}(u),H^{-1}(v),H^{-1}(w)]_\g\\
~ &=&H\circ H^{-1}\Big(D\big(H^{-1}(u),H^{-1}(v)\big)w+\mu\big(H^{-1}(v),H^{-1}(w)\big)u-\mu\big(H^{-1}(u),H^{-1}(w)\big)v+\Courant{u,v,w}_\h\Big)\\
~ &=&D(x,y)H(z)+\mu(y,z)H(x)-\mu(x,z)H(y)+\Courant{Hx,Hy,Hz}_\h,
\end{eqnarray*}
which implies that $H$ is a crossed homomorphism.
\end{proof}

\section{Cohomology of crossed homomorphisms between Lie-Yamaguti algebras}

In this section, we build the cohomology of crossed homomorphisms between Lie-Yamaguti algebras. First of all, we construct a representation of a \LYA ~via a given action.

Let $(\g,[\cdot,\cdot]_\g,\Courant{\cdot,\cdot}_\g)$ and $(\h,[\cdot,\cdot]_\h,\Courant{\cdot,\cdot}_\h)$ be two Lie-Yamaguti algebras, and $H:\g\longrightarrow\h$ a crossed homomorphism with respect to an action $(\rho,\mu)$. Define $\rho_H:\g\longrightarrow\gl(\h),~\mu_H:\otimes^2\g\longrightarrow\gl(\h)$ to be
\begin{eqnarray}
\rho_H(x)u&:=&[Hx,u]_\h+\rho(x)u,\label{newrep1}\\
~\mu_H(x,y)u&:=&\Courant{u,Hx,Hy}_\h+\mu(x,y)u, \quad \forall x,y\in \g,~u\in \h.\label{newrep2}
\end{eqnarray}

\begin{lem}
With the assumptions above, define $D_H:\wedge^2\g\longrightarrow\gl(h)$ to be
\begin{eqnarray}
~D_H(x,y)u&:=&\Courant{Hx,Hy,u}_\h+D(x,y)u,\quad\forall x,y\in \g,~u\in \h.\label{newrep3}
\end{eqnarray}
Then $D_H=D_{\rho_H,\mu_H}$.
\end{lem}
\begin{proof}
For all $x,y\in \g,~u\in \h$, we have
\begin{eqnarray*}
~ &&D_{\rho_H,\mu_H}(x,y)u\\
~ &=&\mu_H(y,x)u-\mu_H(x,y)u+[\rho_H(x),\rho_H(y)](u)-\rho_H([x,y]_\g)u\\
~ &=&\Courant{u,Hy,Hx}_\h+\mu(y,x)u-\Courant{u,Hx,Hy}_\h-\mu(x,y)u\\
~ &&+[Hx,[Hy,u]_\h]_\h+\rho(x)[Hy,u]_\h+[Hx,\rho(y)u]_\h+\rho(x)\rho(y)u\\
~ &&-[Hy,[Hx,u]_\h]_\h-\rho(y)[Hx,u]_\h-[Hy,\rho(x)u]_\h-\rho(y)\rho(x)u\\
~ &&-[H[x,y]_\g,u]_\h-\rho([x,y]_\g)u\\
~ &=&\Courant{u,Hy,Hx}_\h-\Courant{u,Hx,Hy}_\h+[Hx,[Hy,u]_\h]_\h\\
~ &&-[Hy,[Hx,u]_\h]_\h-[[Hx,Hy]_\h,u]_\h+D(x,y)u\\
~ &\stackrel{\eqref{LY1},\eqref{rep}}{=}&\Courant{Hx,Hy,u}_\h+D(x,y)u\\
~ &=&D_H(x,y)u.
\end{eqnarray*}
This completes the proof.
\end{proof}

\begin{pro}
With the assumptions above, then $(\h;\rho_H,\mu_H)$ is a representation of $\g$, where $\rho_H,~\mu_H$, and $D_H$ are given by \eqref{newrep1}-\eqref{newrep3} respectively.
\end{pro}
\begin{proof}
For all $x,y,z\in \g,~u\in \h$, we have
\begin{eqnarray*}
~ &&\mu_H([x,y]_\g,z)u-\mu_H(x,z)\rho_H(y)u+\mu_H(y,z)\rho_H(x)u\\
~ &=&\Courant{u,H[x,y]_\g,Hz}_\h+\mu([x,y],z)u-\Courant{[Hy,u]_\h,Hx,Hz}_\h\\
~ &&-\mu(x,z)[Hy,u]_\h-\Courant{\rho(y)u,Hx,Hz}_\h-\mu(x,z)\rho(y)u\\
~ &&+\Courant{[Hx,u]_\h,Hy,Hz}_\h+\mu(y,z)[Hx,u]_\h+\Courant{\rho(x)u,Hy,Hz}_\h\\
~ &&+\mu(y,z)\rho(x)u\\
~ &\stackrel{\eqref{RLYd}}{=}&\Courant{u,[Hx,Hy]_\h,Hz}_\h-\Courant{[Hy,u]_\h,Hx,Hz}_\h+\Courant{[Hx,u]_\h,Hy,Hz}_\h\\
~ &\stackrel{\eqref{LY2}}{=}&0,
\end{eqnarray*}
and
\begin{eqnarray*}
~ &&\rho_H(\Courant{x,y,z}_\g)u-[D_H(x,y),\rho_H(z)]u\\
~ &=&[H\Courant{x,y,z}_\g,u]_\h+\rho(\Courant{x,y,z}_\g)u-\Courant{Hx,Hy,[Hz,u]_\h}_\h\\
~ &&-D(x,y)[Hz,u]_\h-\Courant{Hx,Hy,\rho(z)u}_\h-D(x,y)\rho(z)u\\
~ &&+[Hz,\Courant{Hx,Hy,u}_\h]_\h+\rho(z)\Courant{Hx,Hy,u}_\h+[Hz,D(x,y)u]_\h\\
~ &&+\rho(z)D(x,y)u\\
~ &\stackrel{\eqref{RLYe}}{=}&[\Courant{Hx,Hy,Hz}_\h,u]_\h-\Courant{Hx,Hy,[Hz,u]_\h}_\h+[Hz,\Courant{Hx,Hy,u}_\h]_\h\\
~ &\stackrel{\eqref{LY3}}{=}&0.
\end{eqnarray*}
Other equalities can be obtained similarly. We omit the details.
\end{proof}

Let $(\g,[\cdot,\cdot]_\g,\Courant{\cdot,\cdot,\cdot}_\g)$ and $(\h,[\cdot,\cdot]_\h,\Courant{\cdot,\cdot,\cdot}_\h)$ be two Lie-Yamaguti algebras. Let $(\rho,\mu)$ be an action of $\g$ on $\h$. Define $\delta:\wedge^2\g\longrightarrow\Hom(\g,\h)$ to be
\begin{eqnarray}
\Big(\delta(x,y)\Big)z:=\mu(y,z)(Hx)-\mu(x,z)(Hy)+\Courant{Hx,Hy,Hz}_\h,\quad\forall x,y,z\in \g.\label{0cochain}
\end{eqnarray}

\begin{pro}\label{0co}
With the notations above, $\delta(x,y)$ defined by \eqref{0cochain} is a $1$-cocycle of the Lie-Yamaguti algebra $(\g,[\cdot,\cdot]_\g,\Courant{\cdot,\cdot,\cdot}_\g)$ with coefficients in the representation $(\h;\rho_H,\mu_H)$.
\end{pro}
\begin{proof}
For all $x_1,x_2,x_3\in \g$, we have
\begin{eqnarray*}
~ &&\delta_{\rm I}\big(\delta(x,y)\big)(x_1,x_2)\\
~ &=&\rho_H(x_1)\delta(x,y)x_2-\rho_H(x_2)\delta(x,y)x_1-\delta(x,y)\big([x_1,x_2]_\g\big)\\
~ &=&\rho_H(x_1)\Big(\mu(y,x_2)(Hx)-\mu(x,x_2)(Hy)+\Courant{Hx,Hy,Hx_2}_\h\Big)\\
~ &&-\rho_H(x_2)\Big(\mu(y,x_1)(Hx)-\mu(x,x_1)(Hy)+\Courant{Hx,Hy,Hx_1}_\h\Big)\\
~ &&-\mu(y,[x_1,x_2]_\g)(Hx)+\mu(x,[x_1,x_2]_\g)(Hy)+\Courant{Hx,Hy,H[x_1,x_2]_\g}_\h\\
~ &=&[Hx_1,\mu(y,x_2)(Hx)]_\h+\rho(x_1)\mu(y,x_2)(Hx)-[Hx_1,\mu(x,x_2)(Hy)]_\h\\
~ &&-\rho(x_1)\mu(x,x_2)(Hy)+[Hx_1,\Courant{Hx,Hy,Hx_2}_\h]_\h+\rho(x_1)\Courant{Hx,Hy,Hx_2}_\h\\
~ &&-[Hx_2,\mu(y,x_1)(Hx)]_\h-\rho(x_2)\mu(y,x_1)(Hx)+[Hx_2,\mu(x,x_1)(Hy)]_\h\\
~ &&+\rho(x_2)\mu(x,x_1)(Hy)-[Hx_2,\Courant{Hx,Hy,Hx_1}_\h]_\h-\rho(x_2)\Courant{Hx,Hy,Hx_1}_\h\\
~ &&-\mu(y,[x_1,x_2]_\g)(Hx)+\mu(x,[x_1,x_2]_\g)(Hy)+\Courant{Hx,Hy,[Hx_1,Hx_2]_\h}_\h\\
&=&0,
\end{eqnarray*}
and
\begin{eqnarray*}
~ &&\delta_{\rm II}\big(\delta(x,y)\big)(x_1,x_2,x_3)\\
~ &=&D_H(x_1,x_2)\delta(x,y)x_3+\mu_H(x_2,x_3)\delta(x,y)x_1-\mu_H(x_1,x_3)\delta(x,y)x_2-\delta(x,y)\Courant{x_1,x_2,x_3}_\g\\
~ &=&D_H(x_1,x_2)\Big(\mu(y,x_3)(Hx)-\mu(x,x_3)(Hy)+\Courant{Hx,Hy,Hx_3}_\h\Big)\\
~ &&+\mu_H(x_2,x_3)\Big(\mu(y,x_1)(Hx)-\mu(x,x_1)(Hy)+\Courant{Hx,Hy,Hx_1}_\h\Big)\\
~ &&-\mu_H(x_1,x_3)\Big(\mu(y,x_2)(Hx)-\mu(x,x_2)(Hy)+\Courant{Hx,Hy,Hx_2}_\h\Big)\\
~ &&-\mu(y,\Courant{x_1,x_2,x_3}_\g)(Hx)+\mu(x,\Courant{x_1,x_2,x_3}_\g)(Hy)-\Courant{Hx,Hy,H\Courant{x_1,x_2,x_3}_\g}_\h\\
~ &=&\Courant{Hx_1,Hx_2,\mu(y,x_3)(Hx)}_\h+D(x_1,x_2)\mu(y,x_3)(Hx)-\Courant{Hx_1,Hx_2,\mu(x,x_3)(Hy)}_\h\\
~ &&-D(x_1,x_2)\mu(x,x_3)(Hy)+\Courant{Hx_1,Hx_2,\Courant{Hx,Hy,Hx_3}_\h}_\h+D(x_1,x_2)\Courant{Hx,Hy,Hx_3}_\h\\
~ &&+\Courant{\mu(y,x_1)(Hx),Hx_2,Hx_3}_\h+\mu(x_2,x_3)\mu(y,x_1)(Hx)-\Courant{\mu(x,x_1)(Hy),Hx_2,Hx_3}_\h\\
~ &&-\mu(x_2,x_3)\mu(x,x_1)(Hy)+\Courant{\Courant{Hx,Hy,Hx_1}_\h,Hx_2,Hx_3}_\h+\mu(x_2,x_3)\Courant{Hx,Hy,Hx_1}_\h\\
~ &&-\Courant{\mu(y,x_2)(Hx),Hx_1,Hx_3}_\h-\mu(x_1,x_3)\mu(y,x_2)(Hx)+\Courant{\mu(x,x_2)(Hy),Hx_1,Hx_3}_\h\\
~ &&+\mu(x_1,x_3)\mu(x,x_2)(Hy)-\Courant{\Courant{Hx,Hy,Hx_2}_\h,Hx_1,Hx_3}_\h-\mu(x_1,x_3)\Courant{Hx,Hy,Hx_2}_\h\\
~ &&-\mu(y,\Courant{x_1,x_2,x_3}_\g)(Hx)+\mu(x,\Courant{x_1,x_2,x_3}_\g)(Hy)-\Courant{Hx,Hy,\Courant{Hx_1,Hx_2,Hx_3}_\h}_\h\\
~ &=&0,
\end{eqnarray*}
which implies that $\delta(x,y)$ is a $1$-cocycle. This finishes the proof.
\end{proof}

By now, we can establish the cohomology of crossed homomorphisms between Lie-Yamaguti algebra as follows.
Let $H:\g\longrightarrow\h$ be a crossed homomorphism from a Lie-Yamaguti algebra $(\g,[\cdot,\cdot]_\g,\Courant{\cdot,\cdot,\cdot}_\g)$ to another \LYA ~$(\h,[\cdot,\cdot]_\h,\Courant{\cdot,\cdot,\cdot}_\h)$ with respect to an action $(\rho,\mu)$. Define the set of $n$-cochains to be
\begin{eqnarray*}
\frkC_H^p(\g,\h)=
\begin{cases}
C_{LieY}^p(\g,\h), & p\geq 1,\\
\wedge^2\g, & p=0.
\end{cases}
\end{eqnarray*}

Define $\partial:\frkC_H^p(\g,\h)\longrightarrow\frkC_H^{p+1}(\g,\h)$ to be
\begin{eqnarray*}
\partial=
\begin{cases}
\delta^H, & p\geq 1,\\
\delta, & p=0,
\end{cases}
\end{eqnarray*}
where the map $\delta^H$ is the corresponding coboundary map given by \eqref{cohomology1}-\eqref{cohomology4} with coefficients in the representation $(\h;\rho_H,\mu_H)$.
Then combining with Proposition \ref{0co}, we obtain that $\Big(\bigoplus_{n=0}^\infty\frkC_H^n(\g,\h),\partial\Big)$ is a complex. Denote the set of $n$-cochains by $\huaZ_H^n(\g,\h)$, and denote the set of $n$-cobonudaries by $\huaB_H^n(\g,\h)$. The resulting $n$-th cohomology group is given by
$$\huaH_H^n(\g,\h):=\huaZ_H^n(\g,\h)/\huaB_H^n(\g,\h), n\geq 0.$$

\begin{defi}
The cohomology of the cochian complex $\Big(\bigoplus_{n=0}^\infty\frkC_H^n(\g,\h),\partial\Big)$ is called the {\bf cohomology} of the crossed homomorphism $H$.
\end{defi}

At the end of this section, we show that a certain homomorphism between two crossed homomorphisms induces a homomorphism between the corresponding cohomology groups. Let $H$ and $H'$ be two crossed homomorphisms from a Lie-Yamaguti algebra $(\g,[\cdot,\cdot]_\g,\Courant{\cdot,\cdot,\cdot}_\g)$ to another Lie-Yamaguti algebra $(\h,[\cdot,\cdot]_\h,\Courant{\cdot,\cdot,\cdot}_\h)$ with respect to an action $(\rho,\mu)$. Let $(\psi_\g,\psi_\h)$ be a homomorphism from $H'$ to $H$, where $\psi_\g$ is invertible. For $n\geqslant 2$, define a linear map $p:\frkC_{H'}^n(\g,\h)\longrightarrow\frkC_{H}^n(\g,\h)$ to be
\begin{eqnarray*}
p_{\rm I}(f)(\frkX_1,\cdots,\frkX_n)&=&\psi_\h\Big(f\big(\psi_\g^{-1}(\frkX_1),\cdots,\psi_\g^{-1}(\frkX_n)\big)\Big),\\
p_{\rm II}(g)(\frkX_1,\cdots,\frkX_n,x)&=&\psi_\h\Big(g\big(\psi_\g^{-1}(\frkX_1),\cdots,\psi_\g^{-1}(\frkX_n),\psi_\g^{-1}(x)\big)\Big), \quad\forall (f,g)\in \frkC_{H'}^n(\g,\h).
\end{eqnarray*}
Here $\frkX_k=x_k\wedge y_k\in \wedge^2\g, ~k=1,2,\cdots,n,~x\in \g$, and we use a notation $\psi_\g^{-1}(\frkX_k)=\psi_\g^{-1}(x_k)\wedge\psi_\g^{-1}(y_k),~k=1,2,\cdots,n.$ In particular, for $n=1$, $p:\frkC_{H'}^1(\g,\h)\longrightarrow\frkC_{H}^1(\g,\h)$ is defined to be
$$p(f)(x)=\psi_\h\Big(f\big(\psi_\g^{-1}(x)\big)\Big), \quad\forall x\in \g, \ f\in \frkC_{H'}^1(\g,\h).$$

\begin{thm}
With the notations above, $p$ is  a cochain map from a cochain $(\oplus_{n=1}^\infty\frkC_{H'}^n(\g,\h),\delta^{H'})$ to $(\oplus_{n=1}^\infty\frkC_{H}^n(\g,\h),\delta^{H})$. Consequently, $p$ induces a homomorphism $p_*:\huaH_{H'}^n(\g,\h)\longrightarrow\huaH_{H}^n(\g,\h)$ between cohomology groups.
\end{thm}
\begin{proof}
For all $(f,g)\in \frkC_{H'}^n(\g,\h)~(n\geq2),$, and for all $\frkX_k=x_k\wedge y_k\in \wedge^2\g,~ k=1,2,\cdots,n,~x\in \g$, we have
\begin{eqnarray*}
~ &&\delta_{\rm II}^H\Big(p_{\rm I}(f),p_{\rm II}(g)\Big)(\frkX_1,\cdots,\frkX_n,x)\\
~ &=&(-1)^{n-1}\Big(\mu_H(y_n,x)p_{\rm II}(g)(\frkX_1,\cdots,\frkX_{n-1},x_n)-\mu_H(x_n,x)p_{\rm II}(g)(\frkX_1,\cdots,\frkX_{n-1},y_n)\Big)\\
~ &&+\sum_{k=1}^n(-1)^{k+1}D_H(\frkX_k)p_{\rm II}(g)(\frkX_1,\cdots,\widehat{\frkX_k},\cdots,\frkX_{n},x)\\
~ &&+\sum_{k<l}^n(-1)^{k}p_{\rm II}(g)(\frkX_1,\cdots,\widehat{\frkX_k},\cdots,\frkX_k\circ\frkX_l,\cdots,\frkX_{n},x)\\
~ &&+\sum_{k=1}^n(-1)^{k}p_{\rm II}(g)(\frkX_1,\cdots,\widehat{\frkX_k},\cdots,\frkX_{n},\Courant{\frkX,x}_\g)\\
~ &=&(-1)^{n-1}\Big(\mu_H(y_n,x)\psi_\h\big(g(\psi_\g^{-1}(\frkX_1),\cdots,\psi_\g^{-1}(\frkX_{n-1}),\psi_\g^{-1}(x_n)\big)\\
~ &&\quad\quad\quad-\mu_H(x_n,x)\psi_\h\big(g(\psi_\g^{-1}(\frkX_1),\cdots,\psi_\g^{-1}(\frkX_{n-1}),\psi_\g^{-1}(y_n)\big)\Big)\\
~ &&+\sum_{k=1}^n(-1)^{k+1}D_H(\frkX_k)\psi_\h\big(g(\psi_\g^{-1}(\frkX_1),\cdots,\widehat{\psi_\g^{-1}(\frkX_k)},\cdots,\psi_\g^{-1}(\frkX_{n}),\psi_\g^{-1}(x)\big)\\
~ &&+\sum_{k<l}(-1)^{k}\psi_\h\big(g(\psi_\g^{-1}(\frkX_1),\cdots,\widehat{\psi_\g^{-1}(\frkX_k)},\cdots,\psi_\g^{-1}(\frkX_k)\circ\psi_\g^{-1}(\frkX_l),
\cdots,\psi_\g^{-1}(\frkX_{n}),\psi_\g^{-1}(x)\big)\\
~ &&+\sum_{k=1}^n(-1)^{k}\psi_\h\big(g(\psi_\g^{-1}(\frkX_1),\cdots,\widehat{\psi_\g^{-1}(\frkX_k)},\cdots,
\psi_\g^{-1}(\frkX_{n}),\psi_\g^{-1}\Courant{\frkX,x}_\g)\big)\\
~ &=&\psi_\h\Big((-1)^{n-1}\Big(\mu_{H'}(\psi_\g^{-1}(y_n),\psi_\g^{-1}(x))g\big(\psi_\g^{-1}(\frkX_1),\cdots,\psi_\g^{-1}(\frkX_{n-1}),\psi_\g^{-1}(x_n)\big)\\
~ &&\quad\quad\quad\quad-\mu_{H'}(\psi_\g^{-1}(x_n),\psi_\g^{-1}(x))g\big(\psi_\g^{-1}(\frkX_1),\cdots,\psi_\g^{-1}(\frkX_{n-1}),\psi_\g^{-1}(y_n)\big)\Big)\Big)\\
~ &&+\psi_\h\Big(\sum_{k=1}^n(-1)^{k+1}D_{H'}(\psi_\g^{-1}(\frkX_k))g\big(\psi_\g^{-1}(\frkX_1),\cdots,\widehat{\psi_\g^{-1}(\frkX_k)},\cdots,
\psi_\g^{-1}(\frkX_{n}),\psi_\g^{-1}(x)\big)\Big)\\
~ &&+\psi_\h\Big(\sum_{k<l}(-1)^{k}g\big(\psi_\g^{-1}(\frkX_1),\cdots,\widehat{\psi_\g^{-1}(\frkX_k)},\cdots,\psi_\g^{-1}(\frkX_k)\circ\psi_\g^{-1}(\frkX_l),
\cdots,\psi_\g^{-1}(\frkX_{n}),\psi_\g^{-1}(x)\big)\Big)\\
~ &&+\psi_\h\Big(\sum_{k=1}^n(-1)^{k}g\big(\psi_\g^{-1}(\frkX_1),\cdots,\widehat{\psi_\g^{-1}(\frkX_k)},\cdots,
\psi_\g^{-1}(\frkX_{n}),\Courant{\psi_\g^{-1}(\frkX),\psi_\g^{-1}(x)}_\g)\big)\Big)\\
~ &=&\psi_\h\Big(\delta_{\rm II}^{H'}(f,g)\big(\psi_\g^{-1}(\frkX_1),\cdots,\psi_\g^{-1}(\frkX_{n}),\psi_\g^{-1}(x)\big)\Big)\\
~ &=&p_{\rm II}\Big(\delta_{\rm II}^{H'}(f,g)\Big)(\frkX_1,\cdots,\frkX_n,x).
\end{eqnarray*}
Note that the third equality holds since for all $x,y_n\in \g,~u\in \h$, by using $(\psi_\g,\psi_\h)$ is a homomorphism from crossed homomorphism $H'$ to crossed homomorphism $H$, we have
\begin{eqnarray*}
~ \mu_H(y_n,x)\psi_\h(u)&=&\Courant{\psi_\h(u),Hy_n,Hx}_\h+\mu(y_n,x)\psi_\h(u)\\
~&=&\Courant{\psi_\h(u),\psi_\h\circ H'\circ \psi_\g^{-1}(y_n),\psi_\h\circ H'\circ \psi_\g^{-1}(x)}_\h+\psi_\h\mu(\psi_\g^{-1}(y_n),\psi_\g^{-1}(x))u\\
~ &=&\psi_\h\Big(\Courant{u,H'\circ \psi_\g^{-1}(y_n),H'\circ \psi_\g^{-1}(x)}_\h+\mu(\psi_\g^{-1}(y_n),\psi_\g^{-1}(x))u\Big)\\
~ &=&\psi_\h\Big(\mu_{H'}(\psi_\g^{-1}(y_n),\psi_\g^{-1}(x))\Big).
\end{eqnarray*}
Thus we obtain that $p_{\rm II}\Big(\delta_{\rm II}^{H'}(f,g)\Big)=\delta_{\rm II}^H\Big(p_{\rm I}(f),p_{\rm II}(g)\Big)$ for all $(f,g)\in \frkC_{H'}^n(\g,\h)$ $(n\geq 2)$. Similarly, we can show that $p_{\rm I}\Big(\delta_{\rm I}^{H'}(f,g)\Big)=\delta_{\rm I}^H\Big(p_{\rm I}(f),p_{\rm II}(g)\Big)$ for all $(f,g)\in \frkC_{H'}^n(\g,\h)$ $(n\geq 2)$. And moreover, it is easy to see that the case of $n=1$ is still valid. This finishes the proof.
\end{proof}

\section{Deformatons of crossed homomorphisms between Lie-Yamaguti algebras}

In this section, we use the cohomology theory established in the former section to characterize deformations of crossed homomorphisms between Lie-Yamaguti algebras.

\subsection{Linear deformations of crossed homomorphisms between Lie-Yamaguti algebras}
In this subsection, we use the cohomology constructed in the former section to characterize the linear deformations of crossed homomorphisms between Lie-Yamaguti algebras.

\begin{defi}
Let $H:\g\longrightarrow\h$ be a crossed homomorphism from a Lie-Yamaguti algebra $(\g,[\cdot,\cdot]_\g,\Courant{\cdot,\cdot,\cdot}_\g)$ to another Lie-Yamaguti algebra $(\h,[\cdot,\cdot]_\h,\Courant{\cdot,\cdot,\cdot}_\h)$ with respect to an action $(\h;\rho,\mu)$. Let $\frkH:\g\longrightarrow\h$ be  a linear map. If $H_t:=H+t\frkH$ is still a crossed homomorphism for all $t$, then we say that $\frkH$ generates a {\bf linear deformation} of the crossed homomorphism $H$.
\end{defi}

It is easy to see that $\frkH$ generates a linear deformation of the crossed homomorphism $H$, then  for all $x,y,z\in \g$, there hold that
\begin{eqnarray*}
\frkH[x,y]_\g&=&\rho(x)(\frkH y)-\rho(y)(\frkH x)+[Hx,\frkH y]_\h+[\frkH x,Hy]_\h,\\
\frkH\Courant{x,y,z}_\g&=&D(x,y)(\frkH z)+\mu(y,z)(\frkH x)-\mu(x,z)(\frkH y)+\Courant{\frkH x,Hy,Hz}_\h\\
~ &&+\Courant{H x,\frkH y,Hz}_\h+\Courant{H x,Hy,\frkH z}_\h,
\end{eqnarray*}
which means that $\frkH$ is a $2$-cocycle of the crossed homomorphism $H$.

\begin{defi}
Let $H:\g\longrightarrow\h$ be a crossed homomorphism from a Lie-Yamaguti algebra $(\g,[\cdot,\cdot]_\g,\Courant{\cdot,\cdot,\cdot}_\g)$ to another Lie-Yamaguti algebra $(\h,[\cdot,\cdot]_\h,\Courant{\cdot,\cdot,\cdot}_\h)$ with respect to an action $(\rho,\mu)$.
 \begin{itemize}
 \item[(i)] Two linear deformations $H_t^1=H+t\frkH_1$ and $H_t^2=H+t\frkH_2$ are called {\bf equivalent} if there exists an element $\frkX\in \wedge^2\g$ such that $({\Id}_\g+t\frkL(\frkX),{\Id}_\h+tD(\frkX))$ is a homomorphism from $H_t^2$ to $H_t^1$.
 \item[(ii)] A linear deformation $H_t=H+t\frkH$ of the crossed homomorphism $H$ is called {\bf trivial} if it is equivalent to $H$.
 \end{itemize}
\end{defi}

If two linear deformations $H_t^2$ and $H_t^1$ are equivalent, then ${\Id}_\g+t\frkL(\frkX)$ is a homomorphism on the Lie-Yamaguti algebra $\g$, which implies that
\begin{eqnarray}
[\Courant{\frkX,y}_\g,\Courant{\frkX,z}_\g]_\g&=&0,\label{nij1}\\
~\Courant{\Courant{\frkX,y}_\g,\Courant{\frkX,z}_\g,t}_\g+\Courant{\Courant{\frkX,y}_\g,z,\Courant{\frkX,t}_\g}_\g+\Courant{y,\Courant{\frkX,z}_\g,\Courant{\frkX,t}_\g}_\g&=&0,\\
~\Courant{\Courant{\frkX,y}_\g,\Courant{\frkX,z}_\g,\Courant{\frkX,t}_\g}_\g&=&0
\end{eqnarray}
for all $y,z,t\in \g$.

Note that by Eqs. \eqref{crosshomo1} and \eqref{crosshomo2}, we obtain that for all $y,z\in \g$,
\begin{eqnarray}
\rho(\Courant{\frkX,y}_\g)D(\frkX)&=&0,\\
~\mu(\Courant{\frkX,y}_\g,z)D(\frkX)+\mu(y,\Courant{\frkX,z}_\g)D(\frkX)+\mu(\Courant{\frkX,y}_\g,\Courant{\frkX,z}_\g)&=&0,\\
~ \mu(\Courant{\frkX,y}_\g,\Courant{\frkX,z}_\g)D(\frkX)&=&0.\label{nij2}
\end{eqnarray}

Moreover, Eq. \eqref{crho1} yields that for all $x,y,z\in \g$,
\begin{eqnarray*}
\Big({\Id}_\h+tD(x,y)\Big)\Big(H+t\frkH_2\Big)z=\Big(H+t\frkH_1\Big)\Big({\Id}_\g+t\frkL(x,y)\Big)z,
\end{eqnarray*}
which implies that
\begin{eqnarray}
\label{cohomology}\frkH_2z-\frkH_1z=\mu(y,z)(H x)-\mu(x,z)(H y)+\Courant{Hx,Hy,Hz}_\h=\partial(\frkX)z.
\end{eqnarray}

Thus by \eqref{cohomology}, we have the following theorem.

\begin{thm}
Let $H:\g\longrightarrow\h$ be a crossed homomorphism from a Lie-Yamaguti algebra $(\g,[\cdot,\cdot]_\g,\Courant{\cdot,\cdot,\cdot}_\g)$ to another Lie-Yamaguti algebra $(\h,[\cdot,\cdot]_\h,\Courant{\cdot,\cdot,\cdot}_\h)$ with respect to an action $(\h;\rho,\mu)$. If two linear deformations $H_t^1=H+t\frkH_1$ and $H_t^2=H+t\frkH_2$ are equivalent, then $\frkH_1$ and $\frkH_2$ are in the same cohomology class in $\huaH_H^2(\g,\h)$.
\end{thm}

\begin{defi}\label{Nijelement}
Let $H:\g\longrightarrow\h$ be a crossed homomorphism from a Lie-Yamaguti algebra $(\g,[\cdot,\cdot]_\g,\Courant{\cdot,\cdot,\cdot}_\g)$ to another Lie-Yamaguti algebra $(\h,[\cdot,\cdot]_\h,\Courant{\cdot,\cdot,\cdot}_\h)$ with respect to an action $(\rho,\mu)$. If an element $\frkX\in \wedge^2\g$ satisfies Eqs. \eqref{nij1}-\eqref{nij2} and the following equality
\begin{eqnarray*}
D(\frkX)\Big(D(\frkX)(Hy)-H[\frkX,y]_\g\Big)=0,\quad\forall y\in \g,
\end{eqnarray*}
then $\frkX$ is called a {\bf Nijenhuis element} associated to $H$. Denote by $\mathsf{Nij}(H)$ the set of Nijenhuis elements associated to $H$.
\end{defi}

It is easy to see that a trivial deformation of a crossed homomorphism between Lie-Yamaguti algebras gives rise to a Nijenhuis element. However, the converse is also true.

\begin{thm}\label{Nijenhuis}
Let $H:\g\longrightarrow\h$ be a crossed homomorphism from a Lie-Yamaguti algebra $(\g,[\cdot,\cdot]_\g,\Courant{\cdot,\cdot,\cdot}_\g)$ to another Lie-Yamaguti algebra $(\h,[\cdot,\cdot]_\h,\Courant{\cdot,\cdot,\cdot}_\h)$ with respect to an action $(\rho,\mu)$. Then for any $\frkX\in \mathsf{Nij}(H)$, $H_t:=H+t\frkH$ with $\frkH=\delta^H(\frkX)$ is a linear deformation of $H$. Moreover, this deformation is trivial.
\end{thm}

We need the following lemma to prove the above theorem.

\begin{lem}\label{lem}
Let $H:\g\longrightarrow\h$ be a crossed homomorphism from $\g$ to $\h$ with respect to an action $(\rho,\mu)$. Let $\psi_\g:\g\longrightarrow\g$ and $\psi_\h:\h\longrightarrow\h$ be two Lie-Yamaguti algebra homomorphisms such  that \eqref{crosshomo1} and \eqref{crosshomo2} hold. Then the linear map $\psi_\h^{-1}\circ H\circ\psi_\g$ is a crossed homomorphism from $\g$ to $\h$ with respect to the action $(\rho,\mu)$.
\end{lem}
\begin{proof}
For all $x,y,z\in \g$, we have
\begin{eqnarray*}
~ &&(\psi_\h^{-1}\circ H\circ\psi_\g)\Big([x,y]_\g\Big)\\
~ &=&\psi_\h^{-1}\Big(\rho(\psi_\g(x))H\big(\psi_\g(y)\big)-\rho(\psi_\g(y))H\big(\psi_\g(x)\big)+[H\circ\psi_\g(x),H\circ\psi_\g(y)]_\h\Big)\\
~ &=&\rho(x)\Big(\psi_\h^{-1}\circ H\circ\psi_\g(y)\Big)-\rho(y)\Big(\psi_\h^{-1}\circ H\circ\psi_\g(x)\Big)+[\psi_\h^{-1}\circ H\circ\psi_\g(x),\psi_\h^{-1}\circ H\circ\psi_\g(y)]_\h,
\end{eqnarray*}
and
\begin{eqnarray*}
~ &&(\psi_\h^{-1}\circ H\circ\psi_\g)\Big(\Courant{x,y,z}_\g\Big)\\
~ &=&\psi_\h^{-1}\Big(D(\psi_\g(x),\psi_\g()y)H\big(\psi_\g(z)\big)+\mu(\psi_\g(y),\psi_\g(z))H\big(\psi_\g(x)\big)-\mu(\psi_\g(x),\psi_\g(z))H\big(\psi_\g(y)\big)\\
~ &&+\Courant{H\circ\psi_\g(x),H\circ\psi_\g(y),H\circ\psi_\g(z)}_\h\Big)\\
~ &=&D(x,y)\Big(\psi_\h^{-1}\circ H\circ\psi_\g(z)\Big)+\mu(y,z)\Big(\psi_\h^{-1}\circ H\circ\psi_\g(x)\Big)-\mu(x,z)\Big(\psi_\h^{-1}\circ H\circ\psi_\g(y)\Big)\\
~ &&+\Courant{\psi_\h^{-1}\circ H\circ\psi_\g(x),\psi_\h^{-1}\circ H\circ\psi_\g(y),\psi_\h^{-1}\circ H\circ\psi_\g(y)}_\h,
\end{eqnarray*}
which implies that the linear map $\psi_\h^{-1}\circ H\circ\psi_\g$ is a crossed homomorphism from $\g$ to $\h$ with respect to the action $(\rho,\mu)$.
\end{proof}

\noindent\emph{The proof of Theorem \ref{Nijenhuis}:} For any Nijenhuis element $\frkX\in \mathsf{Nij}(H)$, we define
$$\frkH:=\delta(\frkX).$$
By Definition \ref{Nijelement}, for any $t$, $H_t=H+t\frkH$ satisfies that
\begin{eqnarray*}
H\circ\Big({\Id}_\g+t\frkL(\frkX)\Big)&=&\Big({\Id}_\h+tD(\frkX)\Big)\circ H_t,\\
~\Big({\Id}_\h+tD(\frkX)\Big)\circ \rho(x)&=&\rho\Big({\Id}_\g+t\frkL(\frkX)(y)\Big)\circ\Big({\Id}_\h+tD(\frkX)\Big),\\
~\Big({\Id}_\h+tD(\frkX)\Big)\circ \mu(y,z)&=&\rho\Big({\Id}_\g+t\frkL(\frkX)(y),{\Id}_\g+t\frkL(\frkX)(z)\Big)\circ\Big({\Id}_\h+tD(\frkX)\Big), \quad\forall y,z\in \g.
\end{eqnarray*}
For $t$ sufficiently small, we see that ${\Id}_\g+t\frkL(\frkX)$ and ${\Id}_\h+tD(\frkX)$ are Lie-Yamaguti algebra homomorphisms. Thus, we have
$$H_t=\Big({\Id}_\h+tD(\frkX)\Big)^{-1}\circ H\circ \Big({\Id}_\g+t\frkL(\frkX)\Big).$$
By Lemma \ref{lem}, we deduce that $H_t$ is a crossed homomorphism from $\g$ to $\h$ for $t$ sufficiently small. Thus $\frkH=\delta(\frkX)$ generates a linear deformations of $H$. It is easy to see that this deformation is trivial. This completes the proof.\qed

\subsection{Formal deformations of crossed homomorphisms between Lie-Yamaguti algebras}
In this subsection, we study formal deformations of crossed homomorphisms between Lie-Yamaguti algebras. Let $\mathbb K[[t]]$ be a ring of power series of one variable $t$. If $(\g,[\cdot,\cdot]_\g,\llbracket\cdot,\cdot,\cdot\rrbracket_\g)$ is a Lie-Yamaguti algebra, then there is a Lie-Yamaguti algebra structure over the ring $\mathbb K[[t]]$ on $\g[[t]]]$ given by
\begin{eqnarray*}
\Big[\sum_{i=0}^\infty x_it^i,\sum_{j=0}^\infty y_jt^j\Big]&=&\sum_{s=0}^\infty\sum_{i+j=s} \Big[x_i,y_j\Big]t^s,\\
~\Courant{\sum_{i=0}^\infty x_it^i,\sum_{j=0}^\infty y_jt^j,\sum_{k=0}^\infty z_kt^k}&=&\sum_{s=0}^\infty\sum_{i+j+k=s}\Courant{x_i,y_j,z_k}t^s,\quad \forall x_i,y_j,z_k\in \g.
\end{eqnarray*}

For any action $(\rho,\mu)$ of a Lie-Yamaguti algebra $(\g,[\cdot,\cdot]_\g,\Courant{\cdot,\cdot,\cdot}_\g)$ on another \LYA ~$(\h,[\cdot,\cdot]_\h,\Courant{\cdot,\cdot,\cdot}_\h)$, there is a natural action of the Lie-Yamaguti algebra $\g[[t]]$ on the  $\mathbb K[[t]]$-Lie-Yamaguti algebra $\h[[t]]$ given by
\begin{eqnarray*}
\rho\Big(\sum_{i=0}^\infty x_it^i\Big)\Big(\sum_{k=0}^\infty v_kt^k\Big)&=&\sum_{s=0}^\infty\sum_{i+k=s}\rho(x_i)v_kt^s,\\
\mu\Big(\sum_{i=0}^\infty x_it^i,\sum_{j=0}^\infty y_jt^j\Big)\Big(\sum_{k=0}^\infty v_kt^k\Big)&=&\sum_{s=0}^\infty\sum_{i+j+k=s}\mu(x_i,x_j)v_kt^s, \quad \forall x_i,y_j\in \g,~v_k\in V.
\end{eqnarray*}

Let $H:\g\longrightarrow\h$ be a crossed homomorphism from a Lie-Yamaguti algebra $(\g,[\cdot,\cdot]_\g,\Courant{\cdot,\cdot,\cdot}_\g)$ to anther Lie-Yamaguti algebra $(\h,[\cdot,\cdot]_\h,\Courant{\cdot,\cdot,\cdot}_\h)$ with respect to an action $(\rho,\mu)$. Consider the power series
\begin{eqnarray}
\label{defor:H}H_t=\sum_{i=0}^\infty\frkH_it^i,\quad \frkH_i\in \Hom(\g,\h),
\end{eqnarray}
that is, $H_t\in \Hom_{\mathbb K}(\g,\h)[[t]]=\Hom_{\mathbb K}(\g,\h[[t]])$.

\begin{defi}
Let $H:\g\longrightarrow\h$ be a crossed homomorphism from a Lie-Yamaguti algebra $(\g,[\cdot,\cdot]_\g,\Courant{\cdot,\cdot,\cdot}_\g)$ to anther Lie-Yamaguti algebra $(\h,[\cdot,\cdot]_\h,\Courant{\cdot,\cdot,\cdot}_\h)$ with respect to an action $(\rho,\mu)$. Suppose that $H_t$ is given by \eqref{defor:H} with $\frkH_0=H$, and $H_t$ also satisfies
\begin{eqnarray}
\label{deforC1}H_t[x,y]_\g&=&\rho(x)(H_ty)-\rho(y)(H_tx)+[H_tx,H_ty]_\h,\\
\label{deforC2}H_t\Courant{x,y,z}_\g&=&D(x,y)(H_tz)+\mu(y,z)(H_tx)-\mu(x,z)(H_ty)+\Courant{H_tx,H_ty,H_tz}_\h,
\end{eqnarray}
for all $x,y,z\in \g$. We say that $H_t$ is a {\bf formal deformation} of $H$.
\end{defi}

Substituting the Eq. \eqref{defor:H} into Eqs. \eqref{deforC1} and \eqref{deforC2} and comparing the coefficients of $t^s~(\forall s\geqslant0)$, we have for all $x,y,z \in \g$,
 \begin{eqnarray}
 \label{sys1}&&\sum_{i+j=s,
 \atop i,j\geqslant0}\Big(\rho(x)(\frkH_{s}y)-\rho(y)(\frkH_{s}x)+[\frkH_ix,\frkH_jy]_\h-\frkH_{s}[x,y]_\g\Big)t^s=0,\\
 \label{sys2}&&\sum_{i+j+k=s,
 \atop i,j,k\geqslant0}\Big(D(x,y)(\frkH_{s}z)+\mu(y,z)(\frkH_{s}x)-\mu(x,z)(\frkH_{s}y)\\
 ~\nonumber &&~~~~~\quad\quad+\Courant{\frkH_ix,\frkH_jy,\frkH_kz}_\h-\frkH_{s}\Courant{x,y,z}_\g\Big)t^s=0.
 \end{eqnarray}

 \begin{pro}\label{formal}
 Let $H:\g\longrightarrow\h$ be a crossed homomorphism from a Lie-Yamaguti algebra $(\g,[\cdot,\cdot]_\g,\Courant{\cdot,\cdot,\cdot}_\g)$ to another Lie-Yamaguti algebra $(\h,[\cdot,\cdot]_\h,\Courant{\cdot,\cdot,\cdot}_\h)$ with respect to an action $(\h;\rho,\mu)$. If $H_t=\sum_{i=0}^\infty\frkH_it^i$ is a formal deformation of $H$, then $\delta^H\frkH_1=0$, i.e., $\frkH_1\in \frkC^1_H(\g,\h)$ is a $1$-cocycle of $H$.
 \end{pro}
 \begin{proof}
 When $s=1$, Eqs. \eqref{sys1} and \eqref{sys2} are equivalent to
 \begin{eqnarray*}
 \frkH_1[x,y]_\g&=&\rho(x)(\frkH_1y)-\rho(y)(\frkH_1x)+[Hx,\frkH_1y]_\h+[\frkH_1x,Hy]_\h,\\
 \frkH_1\Courant{x,y,z}_\g&=&D(x,y)(\frkH_1z)+\mu(y,z)(\frkH_1x)-\mu(x,z)(\frkH_1y)\\
 ~ &&+\Courant{\frkH_1x,Hy,Hz}_\h+\Courant{Hx,\frkH_1y,Hz}_\h+\Courant{Hx,Hy,\frkH_1z}_\h,\quad\forall x,y,z\in \g.
 \end{eqnarray*}
 which implies that $\delta^H(\frkH_1)=0$, i.e., $\frkH_1$ is a $1$-cocycle of $\delta^H$.
 \end{proof}

 \begin{defi}
Let $H:\g\longrightarrow\h$ be a crossed homomorphism from a Lie-Yamaguti algebra $(\g,[\cdot,\cdot]_\g,\Courant{\cdot,\cdot,\cdot}_\g)$ to another Lie-Yamaguti algebra $(\h,[\cdot,\cdot]_\h,\Courant{\cdot,\cdot,\cdot}_\h)$ with respect to an action $(\h;\rho,\mu)$. The $1$-cocycle $\frkH_1$ is called the {\bf infinitesimal} of the formal deformation $H_t=\sum_{i=0}^\infty \frkH_it^i$ of $H.$
 \end{defi}

In the sequel, let us give the notion of equivalent formal deformations of crossed homomorphisms between Lie-Yamaguti algebras.

 \begin{defi}
 Let $H:\g\longrightarrow\h$ be a crossed homomorphism from a Lie-Yamaguti algebra $(\g,[\cdot,\cdot]_\g,\Courant{\cdot,\cdot,\cdot}_\g)$ to another Lie-Yamaguti algebra $(\h,[\cdot,\cdot]_\h,\Courant{\cdot,\cdot,\cdot}_\h)$ with respect to an action $(\rho,\mu)$. Two formal deformations $\bar H_t=\sum_{i=0}^\infty \bar{\frkH_i}t^i$ and $H_t=\sum_{i=0}^\infty \frkH_it^i$, where $\bar{\frkH_0}=\frkH_0=H$ are said to be {\bf equivalent} if there exist $\frkX\in \wedge^2\g,~\phi_i\in \gl(\g)$ and $\varphi_i\in \gl(\h),~i\geqslant2,$ such that for
 \begin{eqnarray}
 \label{equivalent}\phi_t={\Id}_\g+t\frkL(\frkX)+\sum_{i=2}^\infty \phi_it^i,~\quad \varphi_t={\Id}_\h+tD(\frkX)+\sum_{i=2}^\infty \varphi_it^i,
 \end{eqnarray}
 the following conditions are satisfied:
 \begin{eqnarray}
 [\phi_t(x),\phi_t(y)]_\g=\phi_t[x,y]_\g, \quad \Courant{\phi_t(x),\phi_t(y),\phi_t(z)}_\g=\phi_t\Courant{x,y,z}_\g, \quad \forall x,y,z \in \g,
 \end{eqnarray}
 \begin{eqnarray}
 [\varphi_t(u),\varphi_t(v)]_\h=\varphi_t[u,v]_\h, \quad \Courant{\varphi_t(u),\varphi_t(v),\varphi_t(w)}_\h=\varphi_t\Courant{u,v,w}_\h,\quad\forall u,v,w\in \h,
 \end{eqnarray}
 \begin{eqnarray}
 \varphi_t\rho(x)u=\rho(\phi_t(x))(\varphi_t(u)), \quad \varphi_t\mu(x,y)u=\mu(\phi_t(x),\phi_t(y))(\varphi_t(u)), \quad \forall x,y \in \g,~ u \in \h,
 \end{eqnarray}
 and
 \begin{eqnarray}
 \label{eq3}H_t\circ \varphi_t=\phi_t\circ \bar{H_t}
 \end{eqnarray}
 as $\mathbb K[[t]]$-module maps.
 \end{defi}

The following theorem is the second key conclusion in this section.
 \begin{thm}\label{thm2}
 Let $H:\g\longrightarrow\h$ be a crossed homomorphism from a Lie-Yamaguti algebra $(\g,[\cdot,\cdot]_\g,\Courant{\cdot,\cdot,\cdot}_\g)$ to another Lie-Yamaguti algebra $(\h,[\cdot,\cdot]_\h,\Courant{\cdot,\cdot,\cdot}_\h)$ with respect to an action $(\h;\rho,\mu)$. If two formal deformations $\bar H_t=\sum_{i=0}^\infty \bar{\frkH_i}t^i$ and $H_t=\sum_{i=0}^\infty \frkH_it^i$ are equivalent, then their infinitesimals are in the same cohomology classes.
 \end{thm}
 \begin{proof}
 Let $(\phi_t,\varphi_t)$ be the maps defined by \eqref{equivalent}, which makes two deformations $\bar H_t=\sum_{i=0}^\infty \bar{\frkH_i}t^i$ and $H_t=\sum_{i=0}^\infty \frkH_it^i$ equivalent. By \eqref{eq3}, we have
 $$\bar{\frkH_1}z=\frkH_1z+\mu(y,z)(Hx)-\mu(x,z)(Hy)+\Courant{Hx,Hy,Hz}_\h=\frkH_1z+\partial(\frkX)(z),\quad \forall z\in \h,$$
 which implies that $\bar{\frkH}_1$ and $\frkH_1$ are in the same cohomology class.
 \end{proof}

\subsection{Order $n$ deformations of crossed homomorphisms between Lie-Yamaguti algebras}
In this subsection, we introduce a special cohomology class associated to an order $n$ deformation of a crossed homomorphism, and show that a deformation of order $n$ of a crossed homomorphism is extendable if and only if this cohomology class in the second cohomology group is trivial. This is why we call this special cohomology class the obstruction class of a deformation of order $n$ being extendable.

 \begin{defi}
 Let $H:\g\longrightarrow\h$ be a crossed homomorphism from a Lie-Yamaguti algebra $(\g,[\cdot,\cdot]_\g,\Courant{\cdot,\cdot,\cdot}_\g)$ to another Lie-Yamaguti algebra $(\h,[\cdot,\cdot]_\h,\Courant{\cdot,\cdot,\cdot}_\h)$ with respect to an action $(\h;\rho,\mu)$. If $H_t=\sum_{i=0}^n\frkH_it^i $ with  $\frkH_0=H$, $\frkH_i\in \Hom_{\mathbb K}(\g,\h)$, $i=1,2,\cdots,n$, defines a $\mathbb K[t]/(t^{n+1})$-module from $\g[t]/(t^{n+1})$ to the Lie-Yamaguti algebra $\h[t]/(t^{n+1})$ satisfying
 \begin{eqnarray}
 \label{ordern1}H_t[x,y]_\g&=&\rho(x)(H_ty)-\rho(y)(H_tx)+[H_tx,H_ty]_\h,\\
\label{ordern2}H_t\Courant{x,y,z}_\g&=&D(x,y)(H_tz)+\mu(y,z)(H_tx)-\mu(x,z)(H_ty)+\Courant{H_tx,H_ty,H_tz}_\h,
 \end{eqnarray}
 for all $x,y,z\in \g$. we say that $H_t$ is an {\bf order $n$ deformation} of $H$.
 \end{defi}

 \begin{rmk}
 The left hand sides of Eqs. \eqref{ordern1} and \eqref{ordern2} hold in the Lie-Yamaguti algebra $\g[t]/(t^{n+1})$ and the right hand sides of  Eqs. \eqref{ordern1} and \eqref{ordern2} make sense since $H_t$ is a $\mathbb K[t]/(t^{n+1})$-module map.
 \end{rmk}

 \begin{defi}
 Let $H:\g\longrightarrow\h$ be a crossed homomorphism from a Lie-Yamaguti algebra $(\g,[\cdot,\cdot]_\g,\Courant{\cdot,\cdot,\cdot}_\g)$ to another Lie-Yamaguti algebra $(\h,[\cdot,\cdot]_\h,\Courant{\cdot,\cdot,\cdot}_\h)$ with respect to an action $(\rho,\mu)$. Let $H_t=\sum_{i=0}^n\frkH_it^i $ be an order $n$ deformation of $H$. If there exists a $1$-cochain $\frkH_{n+1}\in \Hom_{\mathbb K}(\g,\h)$ such that $\widetilde{H_t}=T_t+\frkH_{n+1}t^{n+1}$ is an order $n+1$ deformation of $H$, then we say that $H_t$ is {\bf extendable}.
 \end{defi}

The following theorem is the third key conclusion in this section.
 \begin{thm}\label{ob}
Let $H:\g\longrightarrow\h$ be a crossed homomorphism from a Lie-Yamaguti algebra $(\g,[\cdot,\cdot]_\g,\Courant{\cdot,\cdot,\cdot}_\g)$ to another Lie-Yamaguti algebra $(\h,[\cdot,\cdot]_\h,\Courant{\cdot,\cdot,\cdot}_\h)$ with respect to an action $(\h;\rho,\mu)$. Let $H_t=\sum_{i=0}^n\frkH_it^i $ be an order $n$ deformation of $H$. Then $H_t$ is extendable if and only if the cohomology class ~$[\Ob^H]\in \huaH_H^2(\g,\h)$ is trivial, where $\Ob^H=(\Ob_{\rm I}^H,\Ob_{\rm II}^H)\in \frkC_H^2(\g,\h)$ is defined by
\begin{eqnarray*}
\Ob_{\rm I}^H(x_1,x_2)&=&\sum_{i+j=n+1,\atop 0<i,j<n+1}[\frkH_ix,\frkH_jy]_\h,\\
\Ob_{\rm II}^H(x_1,x_2,x_3)&=&\sum_{i+j+k=n+1,\atop 0<i.j.k<n+1}\Courant{\frkH_ix_1,\frkH_jx_2,\frkH_kx_3}, \quad \forall x_1,x_2,x_3 \in \g.
\end{eqnarray*}
\end{thm}
\begin{proof}
Let $\widetilde{H_t}=\sum_{i=0}^{n+1}\frkH_it^i$ be an extension of $H_t$, then for all $x,y,z \in \g$,
\begin{eqnarray}
\label{n order1}\widetilde{H_t}[x,y]_\g&=&\rho(x)(\widetilde{H_t}y)-\rho(y)(\widetilde{H_t}x)+[\widetilde{H_t}x,\widetilde{H_t}y]_\h,\\
\label{n order2}\widetilde{H_t}\Courant{x,y,z}_\g&=&D(x,y)(\widetilde{H_t}z)+\mu(y,z)(\widetilde{H_t}x)-\mu(x,z)(\widetilde{H_t}y)
+\Courant{\widetilde{H_t}x,\widetilde{H_t}y,\widetilde{H_t}z}_\h.
\end{eqnarray}
Expanding the Eq.\eqref{n order1} and comparing the coefficients of $t^n$ yields that
\begin{eqnarray*}
\frkH_{n+1}[x,y]_\g=\rho(x)(\frkH_{n+1}y)-\rho(y)(\frkH_{n+1}x)+\sum_{i+j=n+1}[\frkH_ix,\frkH_jy]_\h,
\end{eqnarray*}
which is equivalent to
\begin{eqnarray*}
~ &&\rho(x)(\frkH_{n+1}y)-\rho(y)(\frkH_{n+1}x)+[\frkH_{n+1}x,Hy]_\h+[x,\frkH_{n+1}y]_\h\\
~ &&-\frkH_{n+1}[x,y]_\g+\sum_{i+j=n+1 \atop 0<i,j<n+1}[\frkH_ix,\frkH_jy]_\h=0.
\end{eqnarray*}
i.e.,
\begin{eqnarray}
\Ob_{\rm I}^H+\delta_{\rm I}^T(\frkH_{n+1})=0.\label{ob:cocy1}
\end{eqnarray}
Similarly, expanding the Eq.\eqref{n order2} and comparing the coefficients of $t^n$ yields that
\begin{eqnarray}
\Ob_{\rm II}^H+\delta_{\rm II}^T(\frkH_{n+1})=0.\label{ob:cocy2}
\end{eqnarray}
From \eqref{ob:cocy1} and \eqref{ob:cocy2}, we get
$$\Ob^H=\delta^H(-\frkH_{n+1}).$$
Thus, the cohomology class $[\Ob^H]$ is trivial.

Conversely, suppose that the cohomology class $[\Ob^H]$ is trivial, then there exists $\frkH_{n+1}\in \frkC_H^1(\g,\h)$, such that ~$\Ob^H=-\delta^H(\frkH_{n+1}).$ Set $\widetilde{H_t}=H_t+\frkH_{n+1}t^{n+1}$. Then for all $0 \leqslant s\leqslant n+1$, and for all $x,y,z\in \g$, $\widetilde{H_t}$ satisfies
\begin{eqnarray*}
\sum_{i+j=s}\Big(\frkH_s[x,y]_\g-\big(\rho(x)(\frkH_sy)-\rho(y)(\frkH_sx)+[\frkH_ix,\frkH_jy]_\h\big)\Big)=0,\\
\sum_{i+j+k=s}\Big(\frkH_s\Courant{x,y,z}_\g-\big(D(x,y)(\frkH_sz)+\mu(y,z)(\frkH_sx)-\mu(x,z)(\frkH_sy)+\Courant{\frkH_ix,\frkH_jy,\frkH_kz}_\h\big)\Big)=0,
\end{eqnarray*}
which implies that $\widetilde{H_t}$ is an order $n+1$ deformation of $H$. Hence it is an extension of $H_t$.
\end{proof}

\begin{defi}
Let $H:\g\longrightarrow\h$ be a crossed homomorphism from a Lie-Yamaguti algebra $(\g,[\cdot,\cdot]_\g,\Courant{\cdot,\cdot,\cdot}_\g)$ to another Lie-Yamaguti algebra $(\h,[\cdot,\cdot]_\h,\Courant{\cdot,\cdot,\cdot}_\h)$ with respect to an action $(\h;\rho,\mu)$, and $H_t=\sum_{i=0}^n\frkH_it^i$ be an order $n$ deformation of $H$. Then the cohomology class $[\Ob^H]\in \huaH_H^2(\g,\h)$ defined in Theorem {\rm \ref{ob}} is called the {\bf obstruction class } of $H_t$ being extendable.
\end{defi}

\begin{cor}
Let $H:\g\longrightarrow\h$ be a crossed homomorphism from a Lie-Yamaguti algebra $(\g,[\cdot,\cdot]_\g,\Courant{\cdot,\cdot,\cdot}_\g)$ to another Lie-Yamaguti algebra $(\h,[\cdot,\cdot]_\h,\Courant{\cdot,\cdot,\cdot}_\h)$ with respect to an action $(\h;\rho,\mu)$. If $\huaH_H^2(\g,\h)=0$, then every $1$-cocycle in $\huaZ_H^1(\g,\h)$ is the infinitesimal of some formal deformation of $H$.
\end{cor}

\vspace{2mm}

 \noindent
 {\bf Acknowledgements:}
Qiao was partially supported by NSFC grant 11971282. Xu was partially supported by NSFC grant 12201253 and Natural Science Foundation of Jiangsu Province BK20220510.

\end{document}